\documentclass[a4paper,leqno,10pt]{amsart}

\usepackage{epsf,graphicx,epsfig}
\usepackage{amscd}
\usepackage{amsmath,latexsym,amssymb,amsthm}
\usepackage[nospace,noadjust]{cite}
\usepackage{textcomp}
\usepackage{setspace,cite}
\usepackage{lscape,fancyhdr,fancybox}
\usepackage{stmaryrd}
\usepackage[all,cmtip]{xy}
\usepackage{tikz}
\usepackage{cancel}
\usetikzlibrary{shapes,arrows,decorations.markings}
\usepackage{graphicx}

\usepackage{color}
\usepackage{url}
\usepackage{enumerate}
\usepackage[mathscr]{euscript}
  \theoremstyle{plain}

\swapnumbers
    \newtheorem{thm}{Theorem}[section]
    \newtheorem{prop}[thm]{Proposition}

    \newtheorem{subsec}[thm]{}
\theoremstyle{definition}
    \newtheorem{defn}[thm]{Definition}
        \newtheorem{remark}[thm]{Remark}
\theoremstyle{remark}

\newcommand{\twopartdef}[4]
{
	\left\{
		\begin{array}{ll}
			#1 & \mbox{on } #2 \\
			#3 & \mbox{otherwise} #4
		\end{array}
	\right.
}

\title{}
\author{}
\date{}
\usepackage{amssymb}

\usepackage{hyperref}
\hypersetup{
	colorlinks,
	citecolor=blue,
	filecolor=black,
	linkcolor=blue,
	urlcolor=black
}

\begin{document}
\title[Associative Rota-Baxter algebras]{The $L_\infty$-deformations of associative Rota-Baxter algebras and homotopy Rota-Baxter operators}

\author{Apurba Das}
\address{Department of Mathematics and Statistics,
Indian Institute of Technology, Kanpur 208016, Uttar Pradesh, India.}
\email{apurbadas348@gmail.com}

\author{Satyendra Kumar Mishra}\footnote{Corresponding Author email [A2]: satyamsr10@gmail.com}
\address{Statistics and Mathematics Unit, Indian Statistical Institute Bangalore Centre, Bangalore-560059, Karnataka, India.}
\email{satyamsr10@gmail.com}

%\curraddr{}
%\email{}

\subjclass[2010]{$16$E$40$, $16$S$80$, $16$T$25$, $16$W$99$}

\begin{abstract}
A relative Rota-Baxter algebra is a triple $(A, M, T)$ consisting of an algebra $A$, an $A$-bimodule $M$, and a relative Rota-Baxter operator $T$. Using Voronov's derived bracket and a recent work of Lazarev et al., we construct an $L_\infty [1]$-algebra whose Maurer-Cartan elements are precisely relative Rota-Baxter algebras. By a standard twisting, we define a new $L_\infty [1]$-algebra that controls Maurer-Cartan deformations of a relative Rota-Baxter algebra $(A,M,T)$. We introduce the cohomology of a relative Rota-Baxter algebra $(A, M, T)$ and study infinitesimal deformations in terms of this cohomology (in low dimensions). As an application, we deduce cohomology of coboundary skew-symmetric infinitesimal bialgebras and discuss their infinitesimal deformations. Finally, we define homotopy relative Rota-Baxter operators and find their relationship with homotopy dendriform algebras and homotopy pre-Lie algebras.

\end{abstract}

\maketitle

\noindent

\thispagestyle{empty}

\tableofcontents

\vspace{0.2cm}

\section{Introduction}
Rota-Baxter operators are an algebraic abstraction of integral operators. They first appeared in the fluctuation theory in probability \cite{baxter}. Subsequently, Rota \cite{rota} and Cartier \cite{cart} studied these operators from combinatorial aspects. More precisely, a Rota-Baxter operator on an associative algebra $A$ is a linear map $R : A \rightarrow A$ satisfying 
\begin{align*}
R(a) R(b) = R ( R(a) b + a R(b)), ~ \text{ for } a, b \in A.
\end{align*}
In the last twenty years, several important applications of Rota-Baxter operators were found in the algebraic approach of the renormalization in quantum field theory \cite{conn}, splitting of algebras \cite{aguiar-pre}, Yang-Baxter solutions \cite{aguiar,aguiar-pre}, quasi-symmetric functions \cite{guo-quasi}, and double Lie algebras \cite{double-algebra}. See \cite{guo-book} for more details about Rota-Baxter operators. As a variation of a noncommutative Poisson structure, Uchino  \cite{uchino} defined a generalization of a Rota-Baxter operator in the presence of a bimodule over the algebra. Such operators are called generalized Rota-Baxter operators or relative Rota-Baxter operators (also called $\mathcal{O}$-operators). In this paper, we stick with the term relative Rota-Baxter operator.

\medskip

Rota-Baxter operators and their relative version were also defined on Lie algebras by Kupershmidt \cite{kuper} as an operator analog of classical $r$-matrices. In \cite{laza-rota}, the authors construct an $L_\infty$-algebra which controls the deformation of a relative Rota-Baxter algebra. Subsequently, they find the connection between this $L_\infty$-algebra and the graded Lie algebra obtained in \cite{tang} that governs the deformation of the relative Rota-Baxter operator.

\medskip

Our aim in this paper is to follow the approach of \cite{laza-rota} to study deformations of relative Rota-Baxter (associative) algebras via $L_\infty$-algebras. During the course, we also study deformations of an $\mathsf{AssBimod}$-pair $(A, M)$ consisting of an associative algebra $A$ and an $A$-bimodule $M$. We construct a graded Lie algebra whose  Maurer-Cartan elements are given by $\mathsf{AssBimod}$-pairs. This characterization allows us to construct a differential graded Lie algebra controlling deformations of an $\mathsf{AssBimod}$-pair.

\medskip

Next, using the higher derived brackets of Voronov \cite{voro,voro2}, we construct an $L_\infty$-algebra whose Maurer-Cartan elements are given by relative Rota-Baxter algebras. With this characterization and the well-known construction \cite{getzler} of an $L_\infty$-algebra twisted by a Maurer-Cartan element, we construct a new $L_\infty$-algebra that controls Maurer-Cartan deformations of a relative Rota-Baxter algebra. This new $L_\infty$-algebra also allows us to define the cohomology of a relative Rota-Baxter algebra. Finally, we obtain a long exact sequence connecting the cohomology of the relative Rota-Baxter algebra $(A, M, T)$, cohomology of the underlying $\mathsf{AssBimod}$-pair $(A, M)$ and the cohomology of the relative Rota-Baxter operator $T$ introduced in \cite{das2}.

\medskip

The formal deformation theory of algebraic structures was first developed for associative algebras in the classical work of Gerstenhaber \cite{gers}  and subsequently extended to Lie algebras by Nijenhuis and Richardson \cite{nij-ric}. In \cite{bala}, Balavoine extended the formal deformation theory to algebras over binary quadratic operads. The more general notion of ${\sf R}$-deformations over a local Artinian ring (or a complete local ring) ${\sf R}$ has been described in \cite{doubek-def,fial,kont-soi}. We apply this more general deformation theory to relative Rota-Baxter algebras. We give special attention to infinitesimal deformations of a relative Rota-Baxter algebra and show that there is a one-to-one correspondence between equivalent infinitesimal deformations and elements in the second cohomology space of a relative Rota-Baxter algebra. 
\medskip

Let $A$ be a finite dimemsional associative algebra. An $r$-matrix on $A$ is an element in $A\otimes A$ such that it satisfies associative Yang-Baxter equation on $A$ (\cite{aguiar, aguiar-bial, bai1 }). The relative Rota-Baxter operators on an associative algebra $A$ with respect to coadjoint bimodule $A^*$ are in bijective correspondence with skew-symmetric $r$-matrices on $A$ (see \cite{bai1}). Any skew symmetric $r$-matrix gives a coboundary skew-symmetric  infinitesimal bialgebra \cite{bai1} structure $(A,\Delta_r)$. We define a cochain complex $(C^\bullet(A,\Delta_r),\delta_{inf})$ for coboundary skew-symmetric infinitesimal bialgebras coming from a skew-symmetric $r$-matrix and denote the associated cohomology by $H^\bullet_{inf}(A,\Delta_r)$. We show that there is a long exact sequence of cyclic cohomology spaces of $A^*$, cohomology spaces $H^\bullet_{inf}(A,\Delta_r)$ and the Hochschild cohomology spaces of $A$. Later on, we discuss the infinitesimal deformations in terms of the second cohomology space $H^2_{inf}(A,\Delta_r)$. 

\medskip

In the sequel, we define homotopy relative Rota-Baxter operators on bimodules over $A_\infty [1]$-algebras. More precisely, given an $A_\infty [1]$-algebra and a bimodule over it, we construct an $L_\infty [1]$-algebra using Voronov's construction. This $L_\infty [1]$-algebra is a graded version of a graded Lie algebra constructed in \cite{das2}. We define a homotopy relative Rota-Baxter operator $T$ as a Maurer-Cartan element in the above $L_\infty [1]$-algebra. We also explicitly describe the identities satisfied by the components of $T$. This notion generalizes the strict homotopy relative Rota-Baxter operator defined in \cite{das1}. It is known that a strict homotopy relative Rota-Baxter operator induces a $Dend_\infty [1]$-algebra structure. In \cite{laza-rota}, the authors define homotopy relative Rota-Baxter operators on modules over $L_\infty[1]$-algebras. We show that our definition is suitable compatible with that of \cite{laza-rota}.

\medskip

Finally, we construct a $pre\text{-}Lie_\infty [1]$-algebra associated to a $Dend_\infty  [1]$-algebra, which  generalizes the non-homotopic case \cite{aguiar-pre}. The relationship among dendriform algebras, pre-Lie algebras, associative algebras, and Lie algebras is depicted by the following commutative diagram (see \cite{Chapoton} for details)
\begin{equation}\label{d1}
\begin{CD}
 \text{dendriform algebra} @>>> \text{pre-Lie algebras}   \\
 @VVV @VVV\\
 \text{associative algebras} @>>> \text{Lie algebras}.
\end{CD}
\end{equation}
Moreover, there is a relationship between relative Rota-Baxter algebras, relative Rota-Baxter Lie algebras, dendriform algebras, and pre-Lie algebras that is given by the following commutative diagram (see \cite{das2, tang}) 
\begin{equation}\label{d2}
\begin{CD}
 \text{ relative Rota-Baxter algebras} @>>> \text{dendriform algebras} \\
 @VVV @VVV\\
 \text{ relative Rota-Baxter Lie algebras} @>>> \text{pre-Lie algebra}.
\end{CD}
\end{equation}
We extend both these relationships \eqref{d1} and \eqref{d2} to the homotopy context.
\medskip

This paper is organized as follows. In Section \ref{sec-2}, we recall some basics on relative Rota-Baxter algebras and Voronov's derived bracket construction of $L_\infty$-algebras. Maurer-Cartan characterization of $\mathsf{AssBimod}$-pairs and their cohomology are studied in Section \ref{sec-3}. In the next Section (Section \ref{sec-4}), we construct the $L_\infty$-algebra that controls Maurer-Cartan deformations of a relative Rota-Baxter algebra. Using this, we also define the cohomology of a relative Rota-Baxter algebra. We discuss deformations of a relative Rota-Baxter algebra in Section \ref{sec-5}. In Section \ref{sec-6}, we study cohomology and deformations of coboundary skew-symmetric infinitesimal bialgebra. In the last section, we introduce homotopy relative Rota-Baxter operators and discuss their relationship with homotopy dendriform and homotopy pre-Lie algebras.

\medskip

All (graded) vector spaces, linear maps and tensor products are over a field $\mathbb{K}$ of characteristic $0$. We denote set of all permutations on $k$ elements by $S_k$. A permutation $\sigma \in S_k$ is called an $(i, k-i)$-shuffle if $\sigma(1) < \cdots < \sigma(i)$ and $\sigma (i+1) < \cdots < \sigma (k)$. The set of all $(i, k-i)$-shuffles in $S_k$ is denoted by $S_{(i, k-i)}$. For any permutation $\sigma \in S_k$ and homogeneous elements $v_1, \ldots, v_k$ in a graded vector space $V$, the {\em Koszul sign} $\epsilon (\sigma) = \epsilon (\sigma; v_1, \ldots, v_k ) $ is given by
\begin{align*}
v_{\sigma (1)} \odot \cdots \odot v_{\sigma(k)} = \epsilon (\sigma )~ v_{1} \odot \cdots \odot v_{k},
\end{align*}
where $\odot$ denotes the product in the reduced symmetric algebra $\overline{S}(V)$ of $V$.

\medskip
\section{Preliminaries}\label{sec-2}
In this section, we first recall the Hochschild cohomology \cite{gers} and cyclic cohomology \cite{conn94, loday-cyclic} of associative algebras, and the construction of a graded Lie algebra from \cite{das2}, whose Maurer-Cartan elements are relative Rota Baxter operators. Then, we recall the notion of $V$-data and induced $L_{\infty}$-algebras via higher derived brackets \cite{voro}. 

\subsection{Hochschild and cyclic cohomology of associative algebras}
Let $A=(A,\mu)$ be an associative algebra. We denote the multiplication $\mu:A\otimes A\rightarrow A$ simply by $\mu(a,b)=ab$ for $a,b \in A$. An $A$-bimodule is a vector space $M$ together with two bilinear maps $l:A\otimes M \rightarrow M,~(a,m)\mapsto am$ and $r:M\otimes A \rightarrow M,~(m,a)\mapsto ma$ (called left and right $A$-actions, respectively) satisfying the following identities 
$$(ab)m=a(bm),\quad (ma)b=m(ab),\quad\mbox{and }~~(am)b=a(mb), \quad\mbox{for all  } a,b\in A,~ m\in M.$$
The associative algebra $A$ itself is an $A$-bimodule, where algebra multiplication in $A$ gives the left and right actions on itself. This $A$-bimodule structure on $A$ is called the adjoint $A$-bimodule.

For any vector space $A$, consider the graded vector space $\bigoplus_{n \geq 0} \mathrm{Hom} ( A^{\otimes n+1}, A)$ of multilinear maps. There is a graded Lie bracket, called the {\em Gerstenhaber bracket} \cite{gers} on $\bigoplus_{n \geq 0} \mathrm{Hom} ( A^{\otimes n+1}, A)$ given by $$[f, g] := f \diamond g - (-1)^{mn} g \diamond f, \text{ where }$$
\begin{align*}
(f \diamond g) ( a_1, \ldots, a_{m+n+1}) := \sum_{i=1}^{m+1} (-1)^{(i-1)n} f (a_1, \ldots, a_{i-1}, g ( a_i, \ldots, a_{i+n}), a_{i+n+1}, \ldots, a_{m+n+1}),
\end{align*}
for $f \in \mathrm{Hom}(A^{\otimes m+1}, A)$ and $g \in \mathrm{Hom}(A^{\otimes n+1}, A)$. Let us denote $C^n(A,A):=Hom(A^{\otimes n},A)$ for $n\geq 1$. Then, for the vector space $A$, there is an associated graded Lie algebra $(C^{\bullet+1}(A,A),[~,~])$. Note that $(A, \mu)$ is an associative algebra if and only if $\mu \in \mathrm{Hom}(A^{\otimes 2}, A)$ is a Maurer-Cartan element in the above graded Lie algebra. The coboundary operator 
\begin{equation}\label{Hoch coboundary}
\delta_{Hoch}(f)=(-1)^{n-1}[\mu,f],\quad \mbox{for all } f\in \mathrm{Hom}(A^{\otimes n},A).
\end{equation}
induced from the Maurer-Cartan element $\mu$ is called the Hochschild coboundary operator. The cohomology of the Hochschild cochain complex $\big(C^\bullet_{Hoch}(A):=\bigoplus_{n\geq 1 } \mathrm{Hom}(A^{\otimes n},A),\delta_{Hoch}\big)$ is called the Hochschild cohomology and it is denoted by $H^{\bullet}_{Hoch}(A)$.
Similarly, we can define the Hochschild cochain complex of $A$ with coefficients in $M$. Let us denote it by $C^\bullet_{Hoch}(A,M)$ and denote the  associated cohomology by $H^\bullet_{Hoch}(A,M)$.

\subsubsection*{Cyclic cohomology}
The cyclic cohomology of an associative algebra $A$ is the cohomology of a certain sub-complex of the Hochschild cochain complex $(C^\bullet_{Hoch}(A,A^*),\delta_{Hoch})$, defined as follows: $A^*$ has an $A$-bimodule structure given by $(afb)(c)=f(bac),$ for all $a,b,c \in A$ and $f\in A^*$. Let us consider the Hochschild complex $(C^\bullet_{Hoch}(A,A^*),\delta_{Hoch})$, then the coboundary $\delta_{Hoch}$ induces a differential $b$ on $C^{\bullet}(A):=C^{\bullet+1}(A,k)$, where
\begin{equation}\label{cyclic coboundary}
\begin{split}
&bf(a_1,a_2,\ldots,a_{n+2})\\
&=\sum^{n+1}_{i=1}(-1)^{(i+1)}f(a_1,\ldots, {a_{i-1}},a_i a_{i+1}, {a_{i+2}},\ldots,a_{n+2})+(-1)^{n+1}f(a_{n+2}a_1,a_2,\ldots,a_{n+1}).
\end{split}
\end{equation}  
The cohomology of the complex $(C^\bullet_{Hoch}(A,A^*),\delta_{Hoch})$ is denoted by $HH^\bullet(A)$. Let us consider the Hochschild coboundary $\delta_{Hoch}: C^n(A)\rightarrow C^{n+1}(A)$ of $A$ with trivial coefficients and the cyclic action $\lambda:C^n(A)\rightarrow C^{n}(A)$ defined by
\begin{align*}
\delta_{Hoch}f(a_1,a_2,\ldots,a_{n+2})&=\sum^{n+1}_{i=1}(-1)^{(i+1)}f(a_1,\ldots, {a_{i-1}},a_i a_{i+1}, {a_{i+2}},\ldots,a_{n+2})\\
\lambda f(a_1,a_2,\ldots,a_{n+1})&:=(-1)^n f(a_{n+1},a_1,\ldots,a_{n}).
\end{align*} 
Then, the operators $b$, $\delta_{Hoch}$, and $\lambda$ satisfy the identity $\delta_{Hoch}(1-\lambda)=(1-\lambda)b$. Let us define 
$$C_{\lambda}^\bullet(A):=Ker(1-\lambda),$$
then $(C_{\lambda}^\bullet,b)$ forms a cochain complex and the associated cohomology is called the cyclic cohomology of $A$. Let us denote the cyclic cohomology of $A$ by $HC^\bullet(A)$.
\subsection{Relative Rota-Baxter operators}
\begin{defn}[\cite{uchino}]\label{O-operator}
Let $A$ be an associative algebra and $M$ be an $A$-bimodule. A linear map $T:M\rightarrow A$ is said to be a relative Rota-Baxter operator on $M$ over the algebra $A$ if $T$ satisfies
$$ T(m)T(n)=T(T(m)n+mT(n)),\quad \mbox{for all  } m,n\in M.$$ 
\end{defn}
A Rota-Baxter operator on $A$ is a relative Rota-Baxter operator on the adjoint $A$-bimodule. A relative Rota-Baxter associative algebra is a triple $(A,M,T)$ consisting of an associative algebra $A$, an $A$-bimodule $M$, and a relative Rota-Baxter operator $T:M\rightarrow A$. A Rota-Baxter algebra is a relative Rota-Baxter algebra with respect to the adjoint $A$-bimodule. 

Let $(A,M,T)$ be a relative Rota-Baxter algebra, where the associative product in $A$ is given by $\mu$, and the $A$-bimodule actions on $M$ are given by $l$ and $r$. If we need to emphasize the underlying structure maps, we alternatively use the notation $((A,\mu),(M,l,r),T)$ to denote the relative Rota-Baxter algebra $(A,M,T)$ with structure maps. 

\begin{defn}\label{morphisms}
Let $(A,M,T)$ and $(A^{\prime},M^{\prime},T^{\prime})$ be two relative Rota-Baxter algebras. A homomorphism of relative Rota-Baxter algebras from $(A,M,T)$ to $(A^{\prime},M^{\prime},T^{\prime})$ is a pair $(\varphi,\psi)$ of an algebra homomorphism $\varphi:A\rightarrow A^{\prime}$ and a linear map $\psi:M\rightarrow M^{\prime}$ satisfying 
$$T^{\prime}\circ \psi=\varphi\circ T, \quad \psi(am)=\varphi(a)\psi(m),\quad \mbox{and} \quad \psi(ma)=\psi(m)\varphi(a),\quad \mbox{ for all } a\in A,~m\in M.$$ 
The morphism $(\varphi,\psi)$ is called an isomorphism if $\varphi$ and $\psi$ are linear isomorphisms.
\end{defn}

Let $A$ be an associative algebra and $M$ be an $A$-bimodule. Let $\mu :A\otimes A\rightarrow A$ be the associative product in $A$, and the linear maps $l :A\otimes M\rightarrow M$, $r :M\otimes A\rightarrow M$ be left and right $A$-actions on $M$. Then, the sum $\mu+l+r\in \mathrm{Hom}((A\oplus M)^{\otimes 2}, A\oplus M).$

Let us consider a graded vector space $\bigoplus_{n\geq 1} \mathrm{Hom}(M^{\otimes n}, A)$ with the bracket 
$$[\![P,Q]\!]=(-1)^{m-1}[[\mu+l+r,P],Q],\quad\mbox{for }P\in \mathrm{Hom}(M^{\otimes m}, A)~~\mbox{and}~~ Q\in \mathrm{Hom}(M^{\otimes n}, A).$$
Here, the bracket $[~,~]$ on the right hand side is the Gerstenhaber bracket on the graded vector space $\bigoplus_{n\geq 0}\mathrm{Hom}((A\oplus M)^{\otimes n+1}, A\oplus M)$ (see \cite{das2} for more details). With the above notations, the main result of \cite{das2} can be stated as follows. 
\begin{thm}\label{das-gla}
Let $A$ be an associative algebra and $M$ be an $A$-bimodule. Then,
\begin{enumerate}[(i)]
 \item The graded vector space $ \bigoplus_{n\geq 1} \mathrm{Hom}(M^{\otimes n}, A)$ with the bracket $[\![~,~]\!]$ is a graded Lie algebra. A linear map $T:M\rightarrow A$ is a relative Rota-Baxter operator on $M$ over the algebra $A$ if and only if $T\in Hom(M,A)$ is a Maurer-Cartan element in the graded Lie algebra $(\bigoplus_{n\geq 1} \mathrm{Hom}(M^{\otimes n}, A),[\![~,~]\!])$. Consequently, a relative Rota-Baxter operator $T$ induces a differential $d_T:=[\![T,~]\!]$, which makes the graded Lie algebra $(\bigoplus_{n\geq 1} \mathrm{Hom}(M^{\otimes n}, A),[\![~,~]\!])$ into a differential graded Lie algebra with the differential $d_T$.
 
 \medskip  
\item For a relative Rota-Baxter operator $T:M\rightarrow A$ and a linear map $T^{\prime}:M\rightarrow A$, the sum $T+T^{\prime}$ is also a relative Rota-Baxter operator if and only if $T^{\prime}$ is a Maurer-Cartan element in the differential graded Lie algebra $(\bigoplus_{n\geq 1} \mathrm{Hom}(M^{\otimes n}, A),d_T,[\![~,~]\!])$.  \qed
\end{enumerate}     
\end{thm}

\noindent More precisely, an explicit description of the differential $d_T=[\![T,~]\!]$ is given by 
\begin{equation*}
\begin{split}
d_T P(m_1,m_2,\ldots,m_{n+1})
=&T(P(m_1,m_2,\ldots,m_n)m_{n+1})-(-1)^{n}T(m_1 P(m_2,\ldots,m_{n+1}))\\
 &-(-1)^n\Big\{\sum_{i=1}^{n}(-1)^{i-1}P(m_1,\ldots,m_{i-1},T(m_i)m_{i+1},m_{i+2},\ldots,m_{n+1})\\
 &\quad \quad\quad\quad -\sum_{i=1}^{n} (-1)^{i}  P(m_1,\ldots,m_{i-1},m_iT(m_{i+1}),m_{i+2},\ldots,m_{n+1})\Big\}\\
&+(-1)^n T(m_1)P(m_2,\ldots,m_{n+1})-P(m_1,\ldots,m_n)T(m_{n+1}),
\end{split}
\end{equation*}
for all $P\in \mathrm{Hom}(M^{\otimes n}, A)$ and $m_1,m_2,\ldots,m_{n+1}\in M$.

Let us define a cochain complex $(C^{\bullet}(T),d_T)$ with the graded vector space $C^{\bullet}(T)=\bigoplus_{n\geq 0} C^{n}(T)$, where $C^0(T)=C^1(T):=0$ and  $C^n(T):=\mathrm{Hom}(M^{\otimes n-1}, A)$, for $n\geq 2$. The associated cohomology is called the cohomology of $T$ and denoted by $H^\bullet(T)$.  

\subsection{$L_{\infty}$-algebras and Voronov's construction}
In this subsection, we give necessary background on $L_{\infty}$-algebras and their construction form a $V$-data \cite{voro}. 

The notion of $L_{\infty}$-algebras was introduced by Lada and Stasheff \cite{lada-stasheff} as a homotopy analogy of (graded) Lie algebras. Throughout the paper, we use a `shifted' version of $L_{\infty}$-algebras (called $L_{\infty}[1]$-algebras) in which all multilinear maps are graded symmetric and have degree $1$.
  
\begin{defn}
An $L_{\infty}[1]$-algebra is a graded vector space $W=\sum\limits_{i\in \mathbb{Z}} W_i$ together with a collection of degree $1$ multilinear maps $\{l_k:W^{\otimes k}\rightarrow W\}_{k\geq 1}$ satisfying
\begin{enumerate}[(i)]
\item graded symmetry: for $k\geq 1$,
$$l_k(x_{\sigma(1)},\ldots,x_{\sigma(k)})=\epsilon(\sigma)l_k(x_{1},\ldots,x_k),\quad\mbox{for any }\sigma\in S_k,$$
\item shifted higher Jacobi identities: for each $n\geq 1$,
$$\sum_{i+j=n+1}\sum_{\sigma\in S_{(i,n-i)}}\epsilon(\sigma)l_j(l_i(x_{\sigma(1)},\ldots,x_{\sigma(i)}),x_{\sigma(i+1)},\ldots,x_{\sigma(n)})=0.$$
\end{enumerate}
Here, $S_{(i,n-i)}$ denotes the set of $(i,n-i)$-shuffles in the permutation group $S_{n}$. For any permutation $\sigma\in S_k$ and $x_1,x_2,\ldots,x_k\in V$, $\epsilon(\sigma):=\epsilon(\sigma;x_1,x_2,\ldots,x_k)$ denotes the {\em Koszul sign}.
\end{defn}

\begin{remark}
The equivalence between $L_{\infty}$-algebras and $L_{\infty}[1]$-algebras is given by a degree shift. More precisely, an $L_{\infty}$-algebra structure (in the sense of \cite{sta}) on a graded vector space $V=\sum\limits_{i\in \mathbb{Z}}V_i$ is equivalent to an $L_{\infty}[1]$-algebra structure on the graded vector space $W:=V[1]$, where $W_i=(V[1])_i=V_{i+1}$. The correspondence between multi-brackets of these two structures is given by the d\'ecalage isomorphisms
$$V^{\otimes n}[n]\cong (V[1])^{\otimes n},~~v_1\ldots v_n\mapsto (-1)^{\big((n-1)|v_1|+\cdots+2|v_{n-1}|+|v_{n-1}|\big)}v_1\ldots v_n,$$
where $|v_i|$ denotes the degree of homogeneous element $v_i$ in the graded vector space $V$.
\end{remark}

The notion of filtered $L_{\infty}[1]$-algebras was introduced by Getzler \cite{getzler}, which ensures the convergence of certain infinite sums. For the results in this paper, we only need the following weaker notion \cite{laza-rota}.
\begin{defn}[\cite{laza-rota}]\label{weakly filtered}
A weakly filtered $L_{\infty}[1]$-algebra is a triple $(W,\{l_k\}_{k\geq 1},\mathcal{F}_{\bullet}W)$, where $(W,\{l_k\}_{k\geq 1})$ is an $L_{\infty}[1]$-algebra and $\mathcal{F}_{\bullet}W$ is a descending filtration of $W$ such that $W=\mathcal{F}_{1}W\supset\cdots\supset \mathcal{F}_{n}W\supset \cdots $ and the following conditions are satisfied: 
\begin{enumerate}[(i)]
\item there exists $n\geq 1$ such that $l_k(W,\ldots, W)\subset \mathcal{F}_{k}W,$ for all $k\geq n$ and 
\item $W$ is complete with respect to this filtration. Thus, there is an isomorphism of graded vector spaces $W\cong \lim\limits_{\longleftarrow}W/\mathcal{F}_{n}W$. 
\end{enumerate} 
\end{defn}

\begin{defn}
Let $(W,\{l_k\}_{k\geq 1},\mathcal{F}_{\bullet}W)$ be a weakly filtered $L_{\infty}[1]$-algebra. An element $\alpha\in W^0$ is called a Maurer-Cartan element if $\alpha$ satisfies 
$$\sum_{k=1}^\infty \frac{1}{k!}l_k(\alpha,\alpha,\ldots,\alpha)=0.$$
The set of all Maurer-Cartan elements in $W$ is denoted by $\mathrm{MC}(W)$.
\end{defn}
It is known that a filtered $L_{\infty}[1]$-algebra can be twisted by a Maurer-Cartan element \cite{getzler}. Similarly, for weakly filtered $L_{\infty}[1]$-algebra, we have the following result from \cite{laza-rota}.
\begin{thm}[\cite{laza-rota}]\label{twisted structure}
Let $(W,\{l_k\}_{k\geq 1},\mathcal{F}_{\bullet}W)$ be a weakly filtered $L_{\infty}[1]$-algebra and $\alpha\in W^0$ be a Maurer-Cartan element. Then, $(W,\{l_k^\alpha\}_{k\geq 1},\mathcal{F}_{\bullet}W)$ is a weakly filtered $L_{\infty}[1]$-algebra, where
$$l_k^\alpha(x_1,x_2,\ldots,l_k)=\sum_{n=0}^\infty \frac{1}{n!}l_{k+n}(\underbrace{\alpha,\ldots,\alpha}_n,x_1,x_2,\ldots,x_k). $$\qed  
 \end{thm} 

Next, let us recall the construction of an $L_{\infty}[1]$-algebra from a $V$-data, which is given by Voronov's higher derived brackets \cite{voro}. 
\begin{defn}
A {\em $V$-data} is a quadruple $(L,\mathfrak{a},P,\Delta)$, where $L$ is a graded Lie algebra (with graded Lie bracket $[~,~]$), $\mathfrak{a}\subset L$ is an abelian graded Lie subalgebra, $P:L\rightarrow \mathfrak{a}$ is a projection map whose kernel is a graded Lie subalgebra of $L$, and $\Delta$ is a degree $1$ homogeneous element in $\mathrm{Ker}(P)$ satisfying $[\Delta, \Delta]=0$.
\end{defn}

For a $V$-data $(L,\mathfrak{a},P,\Delta)$, the following theorem yields an $L_{\infty}[1]$-algebra structure on $\mathfrak{a}$. 
\begin{thm}[\cite{voro}]\label{Voro-JPAA1}
Let $(L,\mathfrak{a},P,\Delta)$ be a $V$-data. Then, the graded vector space $\mathfrak{a}\subset L$ together with the operations 
$$l_k(a_1,\ldots,a_k):=P[\cdots[[\Delta,a_1],a_2],\ldots,a_k],\quad\mbox{for } k\geq 1$$ 
is an $L_{\infty}[1]$-algebra.  \qed
\end{thm}
In fact, there is a larger $L_{\infty}[1]$-algebra in which the $L_{\infty}[1]$-algebra $\mathfrak{a}$, defined by Theorem \ref{Voro-JPAA1} is an $L_{\infty}[1]$-subalgebra.

\begin{thm}\label{Voro-JPAA2}
Let $(L,\mathfrak{a},P,\Delta)$ be a $V$-data. Then, the graded vector space $L[1]\oplus\mathfrak{a}$ is an $L_{\infty}[1]$-algebra with multilinear operations 
\begin{align*}
l_1(x[1],a)&=(-[\Delta,x][1],P(x+[\Delta,\mathfrak{a}])),\\
l_2(x[1],y[1])&=(-1)^{|x|}[x,y][1],\\
l_k(x[1],a_1,\ldots,a_{k-1})&=P[\cdots[[x,a_1],a_2],\ldots,a_{k-1}],\quad\mbox{for }k\geq 2,\\
l_k(x[1],a_1,\ldots,a_{k-1})&=P[\cdots[[\Delta,a_1],a_2],\ldots,a_{k-1}],\quad\mbox{for }k\geq 2.
\end{align*}
Here, $x,y$ are homogeneous elements in $L$ and $a,a_1,\ldots,a_k$ are homogeneous elements in $\mathfrak{a}$. Up to the permutations of the above entries, all other multilinear operations vanish.  \qed
\end{thm}

\begin{remark}\label{Voro3}
Let $(L,\mathfrak{a},P,\Delta)$ be a $V$-data. If $L^{\prime}$ is a graded Lie subalgebra of $L$ such that $[\Delta,L^{\prime}]\subset L^{\prime}$, then $L^{\prime}[1]\oplus\mathfrak{a}$ is an $L_{\infty}[1]$-subalgebra of the $L_{\infty}[1]$-algebra $L[1]\oplus\mathfrak{a}$ (see \cite{yael-zam} for more details). 
\end{remark}

\section{cohomology of an $\mathsf{AssBimod}$-pair}\label{sec-3}
Let $A$ be an associative algebra and $M$ be an $A$-bimodule. Then, we call the pair $(A,M)$ an {\em $\mathsf{AssBimod}$-pair}. The aim of this section is to find a graded Lie algebra associated to the vector spaces $A$ and $M$ such that the Maurer-Cartan elements in the graded Lie algebra are in bijective correspondence with the $\mathsf{AssBimod}$-pair structures on the pair $(A,M)$. Subsequently, we define cohomology and deformations of a $\mathsf{AssBimod}$-pair.

First, we recall the following notations form \cite{uchino2}. Let $\mathcal{A}_1$ and $\mathcal{A}_2$ be two vector spaces. For any linear map $f:\mathcal{A}_{i_1}\otimes \mathcal{A}_{i_2}\otimes \cdots\otimes \mathcal{A}_{i_n}\rightarrow \mathcal{A}_{j}$ with $i_1,\ldots,i_n, j\in \{1,2\}$, we define a map $\hat{f}\in \mathrm{Hom}((\mathcal{A}_1\oplus \mathcal{A}_{2})^{\otimes n},\mathcal{A}_1\oplus \mathcal{A}_{2})$ by 
\begin{equation}\label{horizontal lift}
\hat{f}= \twopartdef
{f}      {\mathcal{A}_{i_1}\otimes \mathcal{A}_{i_2}\otimes \cdots\otimes \mathcal{A}_{i_n},}
{0}      {.}
\end{equation}
The map $\hat{f}$ is called a horizontal lift of $f$ or simply a lift of $f$. In particular, the lifts of the linear maps $\mu:\mathcal{A}_1\otimes \mathcal{A}_1\rightarrow \mathcal{A}_1$, $l:\mathcal{A}_1\otimes \mathcal{A}_2\rightarrow \mathcal{A}_2$, and $r:\mathcal{A}_2\otimes \mathcal{A}_1\rightarrow \mathcal{A}_2$ are respectively given by    
\begin{equation*}
\begin{split}
\hat{\mu}((a,u),(b,v))&=(\mu(a,b),0),\\
\hat{l}((a,u),(b,v))&=(0,l(a,v)),\\
\hat{r}((a,u),(b,v))&=(0,r(u,b)),\quad \mbox{for all } a,b\in \mathcal{A}_1, ~~u,v\in \mathcal{A}_2.
\end{split}
\end{equation*}
Let $\mathcal{A}^{k,l}$ be the direct sum of all possible $(k+l)$-tensor powers of $\mathcal{A}_1$ and $\mathcal{A}_2$  in which $\mathcal{A}_1$ (respectively, $\mathcal{A}_2$) appears $k$ times (respectively, $l$ times). For instance, 
$$\mathcal{A}^{2,0}:=\mathcal{A}_1\otimes \mathcal{A}_1,\quad \mathcal{A}^{1,1}:=(\mathcal{A}_1\otimes \mathcal{A}_2)\oplus (\mathcal{A}_2\otimes \mathcal{A}_1), \quad\mathcal{A}^{0,2}:=\mathcal{A}_2\otimes \mathcal{A}_2,$$
and
$$\mathcal{A}^{1,2}:=(\mathcal{A}_1\otimes \mathcal{A}_2\otimes \mathcal{A}_2)\oplus (\mathcal{A}_2\otimes \mathcal{A}_1\otimes \mathcal{A}_2)\oplus (\mathcal{A}_2\otimes \mathcal{A}_2\otimes \mathcal{A}_1).$$ 

Note that there is a vector space isomorphism $(\mathcal{A}_1\oplus \mathcal{A}_2)^{\otimes n}\cong \oplus_{k+l=n} \mathcal{A}^{k,l}$. Therefore, we have the isomorphism (by the horizontal lift)
$$\mathrm{Hom}((\mathcal{A}_1\oplus \mathcal{A}_2)^{\otimes n+1},\mathcal{A}_1\oplus \mathcal{A}_2)\cong \sum\limits_ {k+l=n+1}\mathrm{Hom}(\mathcal{A}^{k,l},\mathcal{A}_1) \oplus \sum_{k+l=n+1}\mathrm{Hom}(\mathcal{A}^{k,l},\mathcal{A}_2)$$

We say that a linear map $f\in \mathrm{Hom}((\mathcal{A}_1\oplus \mathcal{A}_2)^{\otimes n+1},\mathcal{A}_1\oplus \mathcal{A}_2)$ has bidegree $k|l$ with $k+l=n$ if $$f(\mathcal{A}^{k+l,l})\subset \mathcal{A}_1,\quad f(\mathcal{A}^{k,l+l})\subset \mathcal{A}_2, \quad\mbox{and } f \mbox{ is zero otherwise}.$$ Let us denote the bidegree of $f$ by $||f||=k|l$. In general, multilinear maps may not have a bidegree. A multilinear map $f$ is called {\em homogeneous} if $f$ has a bidegree. The set of all homogeneous multilinear maps of bidegree $k|l$ is denoted by $C^{k|l}(\mathcal{A}_1\oplus \mathcal{A}_2,\mathcal{A}_1\oplus \mathcal{A}_2)$. Let us observe that 
\begin{equation*}
\begin{split}
C^{k|0}(\mathcal{A}_1\oplus \mathcal{A}_2,\mathcal{A}_1\oplus \mathcal{A}_2)&\cong \mathrm{Hom}(\mathcal{A}_1^{\otimes k+1},\mathcal{A}_1)\oplus \mathrm{Hom}(\mathcal{A}^{k,1},\mathcal{A}_2),\\
C^{-1|l}(\mathcal{A}_1\oplus \mathcal{A}_2,\mathcal{A}_1\oplus \mathcal{A}_2)&\cong \mathrm{Hom}(\mathcal{A}_2^{l},\mathcal{A}_1).
\end{split}
\end{equation*}

The following result has been proved in \cite[Proposition 2.6]{uchino2}.
\begin{prop}\label{bidegree and G-bracket}
If $f\in \mathrm{Hom}((\mathcal{A}_1\oplus \mathcal{A}_2)^{\otimes m+1},\mathcal{A}_1\oplus \mathcal{A}_2)$ has bidegree $k_f|l_f$ and $g\in \mathrm{Hom}((\mathcal{A}_1\oplus \mathcal{A}_2)^{\otimes n+1},\mathcal{A}_1\oplus \mathcal{A}_2)$ has bidegree $k_g|l_g$, then the Gerstenhaber bracket $[f,g]$ is homogeneous and have the bidegree $(k_f+k_g)|(l_f+l_g)$.   \qed
\end{prop}

Let us recall that $(C^{\bullet+1}(\mathcal{A}_1\oplus \mathcal{A}_2,\mathcal{A}_1\oplus \mathcal{A}_2),[~,~])$ is a graded Lie algebra associated to the vector space $\mathcal{A}_1\oplus\mathcal{A}_2$ with the Gerstenhaber bracket $[~,~]$ . As a consequence of the above proposition, the following result holds.

\begin{prop}\label{the abelian graded Lie subalgebra}
The graded subspace $\bigoplus_{l\geq 1} C^{-1|l}(\mathcal{A}_1\oplus \mathcal{A}_2,\mathcal{A}_1\oplus \mathcal{A}_2)$ is an abelian subalgebra of the graded Lie algebra $C^{\bullet+1}(\mathcal{A}_1\oplus \mathcal{A}_2,\mathcal{A}_1\oplus \mathcal{A}_2)$.
\end{prop}

\begin{proof}
If $||f||=-1|l_f$ and $||g||=-1|l_g$, then by Proposition \ref{bidegree and G-bracket}, we have $[f,g]=0$. Hence, the result follows.
\end{proof}

Also, by a direct application of Proposition \ref{bidegree and G-bracket}, it follows that 
\begin{prop}\label{a graded Lie subalgebra}
The graded vector space $\bigoplus_{k\geq 0} C^{k|0}(\mathcal{A}_1\oplus \mathcal{A}_2,\mathcal{A}_1\oplus \mathcal{A}_2)$ together with the Gerstenhaber bracket is a graded Lie algebra. \qed
\end{prop}

Let $A$ and $M$ be two vector spaces. Suppose there are linear maps $\mu :A\otimes A\rightarrow A$, $l :A\otimes M\rightarrow M$, and $r :M\otimes A\rightarrow M$. Then, all these maps $\textstyle{\mu,l,r\in  C^{1|0}(A\oplus M, A\oplus M)}$, hence the sum 
$$\textstyle{\pi:=\mu+l+r\in C^{1|0}(A\oplus M, A\oplus M)}. $$

The next result shows that $\mathsf{AssBimod}$-pairs are in bijective correspondence with Maurer-Cartan elements in the graded Lie algebra $\textstyle{\bigoplus_{k\geq 0} C^{k|0}(A\oplus M, A\oplus M)}$. 

\begin{prop}\label{AssBimod-MC}
With the above notations, $\mu$ defines an associative product on $A$ and $l,~r$ defines an $A$-bimodule structure on $M$ if and only if $\textstyle{\pi=\mu+l+r\in C^{1|0}(A\oplus M, A\oplus M)}$ is a Maurer-Cartan element in the graded Lie algebra $\textstyle{\bigoplus_{k\geq 0} C^{k|0}(A\oplus M, A\oplus M)}$.
\end{prop} 
\begin{proof}
The proof of the proposition is similar to \cite[Proposition $2.11$]{das2}.
\end{proof}
It follows that a $\mathsf{AssBimod}$-pair $(A,M)$ induces a differential $d_{\mu+l+r}:=[\mu+l+r,~]$ on the graded vector space $\textstyle{\bigoplus_{k\geq 0} C^{k|0}(A\oplus M, A\oplus M)}$, which makes $\textstyle{(\bigoplus_{k\geq 0} C^{k|0}(A\oplus M, A\oplus M),d_{\mu+l+r},[~,~])}$ into a differential graded Lie algebra. Moreover, using the standard Maurer-Cartan theory, we can deduce the following theorem.

\begin{thm}
Let $(A,M)$ be an $\mathsf{AssBimod}$-pair and $\textstyle{\mu+l+r\in C^{1|0}(A\oplus M, A\oplus M)}$ be the corresponding Maurer-Cartan element. Then, for linear maps $\mu^{\prime}:A\otimes A\rightarrow A$, $l^{\prime}:A\otimes M\rightarrow M$, and $r^{\prime}:M\otimes A\rightarrow M$, the sum $(\mu+\mu^{\prime})+(l+l^{\prime})+(r+r^{\prime})$ corresponds to an $\mathsf{AssBimod}$-pair structure on $(A,M)$ if and only if $\mu^{\prime}+l^{\prime}+r^{\prime}$ is a Maurer-Cartan element in the differential graded Lie algebra $\textstyle{(\bigoplus_{k\geq 0} C^{k|0}(A\oplus M, A\oplus M),d_{\mu+l+r},[~,~])}$.  \qed
\end{thm}

We now define the cohomology of an $\mathsf{AssBimod}$-pair $(A,M)$. Let us define the $0$-th cochain group $\textstyle{\mathcal{C}^0(A,M)}$ to be $0$ and for $n\geq 1$, the $n$-th cochain group
$$\mathcal{C}^n(A,M)=C^{(n-1)|0}(A\oplus M, A\oplus M)=\mathrm{Hom}(A^{\otimes n}, A)\oplus \mathrm{Hom}(\mathcal{A}^{n-1,1},M).$$ 
Here, $\mathcal{A}^{n-1,1}$ is the direct sum of all possible $n$-tensor powers of $A$ and $M$ in which $A$ appears $n-1$ times (and thus, $M$ appears only once). Finally, we define a differential $\partial:\mathcal{C}^n(A,M)\rightarrow\mathcal{C}^{n+1}(A,M)$, which is induced from the Maurer-Cartan element, i.e.
\begin{equation}\label{partial}
\partial f=(-1)^{n-1}[\mu+l+r,f],\quad \mbox{for } f\in\mathcal{C}^n(A,M).
\end{equation}
The cohomology of the cochain complex $\big(\mathcal{C}^{\bullet}(A,M),\partial\big)$ is called the cohomology of the $\mathsf{AssBimod}$-pair $(A,M)$.

\section{$L_\infty$-algebras associated to relative Rota-Baxter algebras}\label{sec-4}

\medskip

Let $A$ and $M$ be two vector spaces. Then, $L = ( \bigoplus_{n \geq 0} \mathrm{Hom} ( (A \oplus M )^{\otimes n+1}, A \oplus M), [~,~])$ is a graded Lie algebra where the graded Lie bracket is given by the Gerstenhaber bracket. It has been observed in Proposition \ref{the abelian graded Lie subalgebra} that the graded subspace 
\begin{align*}
 \mathfrak{a} = \bigoplus_{n \geq 0} C^{-1 | n+1 } (A \oplus M, A \oplus M) = \bigoplus_{n \geq 0} \mathrm{Hom} ( M^{\otimes n +1}, A)
\end{align*} 
is an abelian subalgebra of $L$. Let $P : L \rightarrow \mathfrak{a}$ be the projection onto the subspace $\mathfrak{a}$ and take $\triangle = 0 \in \mathrm{ker} (P)_1$. It is also easy to see that the kernel of $P$ is a graded Lie subalgebra of $L$. Hence, with all these notations, the quadruple $(L, \mathfrak{a}, P, \triangle)$ is a $V$-data. Thus, Theorem \ref{Voro-JPAA2} yields the following result.

\begin{thm}\label{infinity1}
There is an $L_\infty [1]$-algebra on the graded vector space $L[1] \oplus \mathfrak{a}$ with structure maps
\begin{align*}
l_1 ( Q[1], \theta) =~& P (Q),\\
l_2 ( Q[1], Q'[1]) =~& (-1)^{|Q|} [Q, Q'][1],\\
l_k (Q[1], \theta_1, \ldots, l_{k-1}) =~& P [ \cdots [[Q, \theta_1], \theta_2],\ldots, \theta_{k-1}],
\end{align*}
for homogeneous elements $Q, Q' \in L$ and $\theta, \theta_1, \ldots, \theta_{k-1} \in \mathfrak{a}$.   \qed
\end{thm}

\begin{remark}\label{infinity2}
From Proposition \ref{a graded Lie subalgebra}, $L' = (\bigoplus_{n \geq 0} C^{n|0} (A \oplus M, A \oplus M), [~,~])$ is a graded Lie subalgebra of the graded Lie algebra $L$. Now, Remark \ref{Voro3} implies that there is an $L_\infty [1]$-algebra on the graded vector space $L'[1] \oplus \mathfrak{a}$, whose structure maps are given by 
\begin{align*}
l_2 ( Q[1], Q'[1]) =~& (-1)^{|Q|} [Q, Q'][1],\\
l_k (Q[1], \theta_1, \ldots, l_{k-1}) =~& P [ \cdots [[Q, \theta_1],\theta_2], \ldots, \theta_{k-1}],
\end{align*}
for homogeneous elements $Q, Q' \in L'$ and $\theta, \theta_1, \ldots, \theta_{k-1} \in \mathfrak{a}$. Furthermore, this $L_\infty$-algebra is weakly filtered with $n = 3$ and the filtration is given by
\begin{align*}
\mathcal{F}_1 = L'[1] \oplus \mathfrak{a}, ~~~ \mathcal{F}_2 = P [ L'[1] \oplus \mathfrak{a} , \mathfrak{a}] \text{ and } ~~ \mathcal{F}_k = P [ \cdots [ L'[1] \oplus \mathfrak{a}, \mathfrak{a} ], \ldots, \mathfrak{a}] ~ \text{ for } k \geq 3.
\end{align*}
\end{remark}

Note that, the degree $0$ component of the graded vector space $L'[1] \oplus \mathfrak{a}$ is given by
\begin{align*}
(L'[1] \oplus \mathfrak{a})_0 =~& L'_1 \oplus \mathfrak{a}_0 \\
=~& \underbrace{\mathrm{Hom} ( A^{\otimes 2}, A ) \oplus \mathrm{Hom} (A \otimes M, M ) \oplus \mathrm{Hom}(M \otimes A, M )}_{L'_1} \oplus \underbrace{\mathrm{Hom}(M, A)}_{ \mathfrak{a}_0}. 
\end{align*}

Let $A$ and $M$ be two vector spaces. Suppose there are maps
\begin{align}\label{many-maps}
\mu \in \mathrm{Hom}(A^{\otimes 2}, A), \quad l \in  \mathrm{Hom}(A \otimes M, M), \quad r \in \mathrm{Hom}(M \otimes A , M) \quad \text{ and } \quad T \in \mathrm{Hom}( M , A).
\end{align}
Consider the element $\pi = \mu + l + r \in C^{1|0} (A \oplus M, A \oplus M) = L'_1$. Then we have $(\pi[1], T) \in (L'[1] \oplus \mathfrak{a})_0$.

\begin{thm}
With the above notations, the triple $((A, \mu), (M, l, r), T)$ is a relative Rota-Baxter algebra if and only if $(\pi [1], T)$ is a Maurer-Cartan element in the $L_\infty[1]$-algebra $(L'[1] \oplus \mathfrak{a}, \{ l_k \}_{k \geq 1})$.
\end{thm}

\begin{proof}
The triple $((A, \mu), (M, l, r), T)$ is a relative Rota-Baxter algebra if and only if $(A,M)$ is a $\mathsf{AssBimod}$-pair and $T:M\rightarrow A$ is a relative Rota-Baxter operator on $M$ over the algebra $A$. Equivalently, by Theorem \ref{das-gla} and Proposition \ref{AssBimod-MC}, $((A, \mu), (M, l, r), T)$ is a relative Rota-Baxter algebra if and only if
\begin{equation}\label{mc:cond1}
[\mu+l+r,\mu+l+r]=0\quad \mbox{and }\quad [\![T,T]\!]=[[\mu+l+r,T],T]=0.
\end{equation}

If $\pi:=\mu+l+r$, then Remark \ref{infinity2} implies that 
\begin{equation}\label{mc:cond2}
\begin{split}
l_2((\pi[1],T),(\pi[1],T))&=-[\pi,\pi][1],\\
l_3((\pi[1],T),(\pi[1],T),(\pi[1],T))&=[[\pi,T],T]=[\![T,T]\!].
\end{split}
\end{equation}

Let us consider the sum 
\begin{equation*}
\begin{split}
\sum_{k=1}^\infty \frac{1}{k!}l_k((\pi[1],T),(\pi[1],T),\ldots,(\pi[1],T))&=\frac{1}{2}l_2((\pi[1],T),(\pi[1],T))+\frac{1}{6}
l_3((\pi[1],T),(\pi[1],T),(\pi[1],T))\\
&=(-\frac{1}{2}[\pi,\pi][1],\frac{1}{6}[[\pi,T],T])=(-\frac{1}{2}[\pi,\pi][1],\frac{1}{6}[\![T,T]\!])
\end{split}
\end{equation*}
From equation \eqref{mc:cond1} and equation \eqref{mc:cond2}, it is clear that $((A, \mu), (M, l, r), T)$ is a relative Rota-Baxter algebra if and only if 
$$\sum_{k=1}^\infty \frac{1}{k!}l_k((\pi[1],T),(\pi[1],T),\ldots,(\pi[1],T))=0.$$
Hence, the theorem holds true. 
\end{proof}

The above theorem implies that relative Rota-Baxter algebras can be interpreted as Maurer-Cartan elements in a $L_\infty [1]$-algebra. This allows us to construct a new $L_\infty[1]$-algebra twisted by the Maurer-Cartan element associated with a given relative Rota-Baxter algebra. More precisely, let $((A, \mu), (M, l, r), T)$ be a relative Rota-Baxter algebra. Consider the Maurer-Cartan element $\alpha = (\pi [1], T) \in (L'[1] \oplus \mathfrak{a})_0$ in the $L_\infty [1]$-algebra $(L'[1] \oplus \mathfrak{a}, \{ l_k \}_{k \geq 1})$. By applying Theorem \ref{twisted structure} to this case, we get the twisted $L_\infty [1]$-algebra that controls deformations of the relative Rota-Baxter algebra.

\begin{thm}\label{thm-twist-das}
Let $((A, \mu), (M, l, r), T)$ be a relative Rota-Baxter algebra. Then there is a twisted $L_\infty [1]$-algebra $(L'[1] \oplus \mathfrak{a}, \{ l_k^{   (\pi [1], T)  } \}_{k \geq 1} ).$ Moreover, for any linear maps $\mu' , l', r', T' $ as of (\ref{many-maps}), the triple $$((A, \mu + \mu'), (M, l+l', r+ r'), T + T')$$ is a relative Rota-Baxter algebra if and only if $(\pi' [1], T')$ is a Maurer-Cartan element in the $L_\infty [1]$-algebra $(L'[1] \oplus \mathfrak{a}, \{ l_k^{   (\pi [1], T)  } \}_{k \geq 1} ),$ where $\pi' = \mu'+ l'+r' \in L'_1$. \qed
\end{thm}

\begin{remark}\label{ses1}
Let $A$ be an associative algebra and $M$ be an $A$-bimodule. Suppose $T: M \rightarrow A$ is a relative Rota-Baxter operator which makes the triple $(A, M, T)$ a relative Rota-Baxter algebra. We have already seen various graded Lie algebras (hence $L_\infty [1]$-algebras by degree shift) and $L_\infty[1]$-algebras.

$\blacktriangleright$ In Proposition \ref{a graded Lie subalgebra}, we find a graded Lie algebra structure on $\bigoplus_{n \geq 0} C^{n|0} (A \oplus M , A \oplus M)$, hence an $L_\infty[1]$-algebra on $\big( \bigoplus_{n \geq 0} C^{n|0} (A \oplus M , A \oplus M)\big)[1]$.

\medskip

$\blacktriangleright$ In Theorem \ref{das-gla}, we have a graded Lie algebra structure on $\bigoplus_{n \geq 1} \mathrm{Hom} (M^{\otimes n}, A)$, hence we get an $L_\infty[1]$-algebra.

\smallskip

$\blacktriangleright$ Finally, in Theorem  \ref{thm-twist-das}, we obtain an $L_\infty$-algebra $(L'[1] \oplus \mathfrak{a}, \{ l_k^{   (\pi [1], T)  } \}_{k \geq 1} )$.

Similar to the Lie algebra case \cite[Theorem 3.16]{laza-rota}, one can prove that there is a short exact sequence of $L_\infty[1]$-algebras
\begin{align*}
\xymatrix{
0 \ar[r] & \big( \bigoplus\limits_{n \geq 0} \mathrm{Hom} (M^{\otimes n+1}, A)\big) \ar[r]^{\quad \qquad i} & L'[1] \oplus \mathfrak{a}  \ar[r]^{p \quad \quad \quad \qquad} & \big( \bigoplus\limits_{n \geq 0} C^{n|0} (A \oplus M , A \oplus M)\big)[1] \ar[r] & 0
}
\end{align*}
by strict morphisms of $L_\infty[1]$-algebras, where $i(P) = (0, P)$ and $p (f[1], P) = f[1]$ for $P\in Hom(M^{\otimes n+1},A)$ and $f\in C^{n|0}(A\oplus M, A\oplus M)$.  
\end{remark}

\subsection{Cohomology of relative Rota-Baxter algebras} Using the twisted $L_\infty[1]$-algebra constructed in Theorem \ref{thm-twist-das}, we define the cohomology of a relative Rota-Baxter algebra $(A, M, T)$. This cohomology is related to the cohomology of the relative Rota-Baxter operator $T$ and the cohomology of the AssBimod pair $(A, M)$ by a long exact sequence.

Let $((A, \mu), (M, l, r), T)$ be a relative Rota-Baxter algebra. Consider the twisted $L_\infty[1]$-algebra 
\begin{align*}
(L'[1] \oplus \mathfrak{a}, \{ l_k^{(\pi[1], T)} \}_{k \geq 1})
\end{align*}
as given in Theorem \ref{thm-twist-das}. Then it follows from the definition of an $L_\infty[1]$-algebra that $l_1^{(\pi[1], T)}  \circ l_1^{(\pi[1], T)} = 0$. In other words,  $l_1^{(\pi[1], T)}$ is a differential. In the following, we will use this differential to define the cohomology of the relative Rota-Baxter algebra $(A, M, T)$.

Define the space of $0$-cochains $C^0 (A, M, T)$ to be $0$, and the space of $1$-cochains $C^1(A, M, T)$ to be $\mathrm{Hom}(A, A ) \oplus \mathrm{Hom}(M,M)$. For $n \geq 2$, the space of $n$-cochains $C^n(A, M, T)$ to be defined by
\begin{align*}
C^n (A, M, T) =~& \mathcal{C}^n (A, M) \oplus C^n (T) \\
=~& \big(    \mathrm{Hom} (A^{\otimes n}, A) \oplus \mathrm{Hom}( \mathcal{A}^{n-1, 1}, M) \big) \oplus \mathrm{Hom} ( M^{\otimes n-1}, A).
\end{align*}

To define the coboundary, we observe that if $(f, P) \in C^n (A, M, T)$, then $(f[1], P ) \in (L'[1] \oplus \mathfrak{a})_{n-2}$. We define the coboundary $\mathcal{D} : C^n (A, M, T) \rightarrow C^{n+1} (A, M, T)$ by
\begin{align*}
\mathcal{D} (f, P ) = (-1)^{n-2}~ l_1^{(\pi[1], T)} ( f[1], P).
\end{align*}

\noindent Then, it follows that $\mathcal{D}^2 =0$. The explicit description of $\mathcal{D}$ is given by
\begin{align*}
\mathcal{D}(f, P) = (-1)^{n-2} \big( - [\pi, f], ~ [[\pi, T], P] + \frac{1}{n !} \underbrace{[ \cdots [[}_nf, T],T], \ldots, T] \big).
\end{align*}

It follows that the coboundary map $\mathcal{D} : C^n (A,M, T) \rightarrow C^{n+1}(A,M, T)$ can also be written as
\begin{align*}
\mathcal{D} (f, P) = ( \partial f, (-1)^{n} d_T P + \Omega f),
\end{align*}
where $\partial f=[\mu+l+r,f]$ (given by equation \eqref{partial}), $d_TP=[\![T,P]\!]$ (given in Theorem \ref{das-gla}), and $\Omega$ is defined by
\begin{align*}
(\Omega f)(a_1, \ldots, a_n) = (-1)^n \big(   f (Ta_1, \ldots, Ta_n) - \sum_{i=1}^n Tf (Ta_1, \ldots, a_i, \ldots, Ta_n)).
\end{align*}

Therefore, $(C^\bullet (A, M, T), \mathcal{D})$ is a cochain complex. The cohomology of this complex is called the cohomology of the relative Rota-Baxter algebra $(A, M, T)$ and denoted by $H^\bullet (A, M, T).$ From the definition of the coboundary map $\mathcal{D}$, we have the following short exact sequence of cochain complexes
\begin{align*}
\xymatrix{
0 \ar[r] & (C^\bullet (T), \delta_T) \ar[r]^i & ((C^\bullet (A, M, T), \mathcal{D}) \ar[r]^p & (\mathcal{C}^\bullet (A, M ), \partial) \ar[r] & 0,
}
\end{align*}
where $i(P) = (0, P)$ and $p (f,P) = f$ for $P\in C^\bullet(T)$ and $f\in \mathcal{C}^\bullet (A,M)$. Therefore, we get the following result.

\medskip
\begin{thm}\label{long-exact-seq-thm}
Let $(A, M, T)$ be a relative Rota-Baxter algebra. Then there is a long exact sequence of cohomology spaces
\begin{align*}
\xymatrix{
\cdots \ar[r] & H^n(T) \ar[r]^{H^n (i)} & H^n (A, M, T) \ar[r]^{H^n (p)} & H^n (A, M ) \ar[r]^{\delta^n} & H^{n+1} (T) \ar[r] & \cdots, 
}
\end{align*}
where the connecting homomorphism $\delta^n $ is given by $\delta^n ([f]) = [d_T (f)]$, for $[f] \in H^n (A, M)$.\qed
\end{thm}

\subsection{Cohomology of Rota-Baxter algebras}
In this subsection, we focus on Rota-Baxter algebras, i.e. relative Rota-Baxter algebras with respect to the adjoint representation. We deduce  cohomology of Rota-Baxter algebras following the cohomology of relative Rota-Baxter algebras modulo certain modifications.

Let $(A, T)$ be a Rota-Baxter algebra. We define the space of $0$-cochains $C^0_{\mathrm{RB}} (A, T)$ to be $0$, and the space of $1$-cochains $C^1_{\mathrm{RB}}(A, T)$ to be $\mathrm{Hom}(A, A)$. The space of $n$-cochains $C^n_{\mathrm{RB}} (A, T)$, for $n \geq 2$, is defined by
\begin{align*}
C^n_{\mathrm{RB}}(A, T) := \mathrm{Hom}(A^{\otimes n}, A) \oplus \mathrm{Hom}(A^{\otimes n-1}, A). 
\end{align*}
We define the coboundary map using the following inclusion map
\begin{align*}
i : C^n_{\mathrm{RB}} (A, T) \rightarrow C^n(A, A, T), ~ (f, P) \mapsto (f, f, P).
\end{align*}

Here, $C^n (A, A, T)$ is the $n$-th cochain group of the relative Rota-Baxter algebra $(A, A, T)$.
Let us write an element $f\in \mathrm{Hom} (A^{\otimes n}, A) \oplus \mathrm{Hom}( \mathcal{A}^{n-1, 1}, M)$ as a pair $(f_A,f_M)$ with the components $f_A\in \mathrm{Hom} (A^{\otimes n}, A)$  and $f_M \in \mathrm{Hom}( \mathcal{A}^{n-1, 1}, M)$. Then, we can write the differential $\partial:\mathrm{Hom} (A^{\otimes n}, A) \oplus \mathrm{Hom}( \mathcal{A}^{n-1, 1}, M) \rightarrow \mathrm{Hom} (A^{\otimes n+1}, A) \oplus \mathrm{Hom}( \mathcal{A}^{n, 1}, M)$ as follows
$$\partial(f)=\big((\partial f)_A,(\partial f)_M\big).$$

In the case, when $M=A$ is the adjoint $A$-bimodule, it is clear that 
$$(\partial f)_A=(\partial f)_M=(-1)^{n-1}[\mu,f]=\delta_{Hoch}f.$$  
Thus, we have the following proposition.

\begin{prop}
The complex $(C^\bullet_{\mathrm{RB}} (A, T), \mathcal{D})$ is a subcomplex of the cochain complex $(C^\bullet (A,A, T), \mathcal{D})$ associated to the relative Rota-Baxter algebra $(A, A, T).$ \qed
\end{prop}

The induced coboundary map $\mathcal{D} : C^n_{\mathrm{RB}} (A, T) \rightarrow C^{n+1}_{\mathrm{RB}} (A, T)$ is explicitly given by
\begin{align*}
\mathcal{D} (f, P) = ( \delta_{\mathrm{Hoch}} f, ~ (-1)^n d_T P + \Omega f),
\end{align*}
where $\delta_{hoch}$ is the Hochschild coboundary (given by equation \eqref{Hoch coboundary}), $d_TP=[\![T,P]\!]$ (given in Theorem \ref{das-gla}), and $\Omega$ is given by
\begin{align*}
(\Omega f)(a_1, \ldots, a_n) = (-1)^n \big(   f (Ta_1, \ldots, Ta_n) - \sum_{i=1}^n Tf (Ta_1, \ldots, a_i, \ldots, Ta_n)).
\end{align*}
The cohomology of the cochain complex $(C^\bullet_{\mathrm{RB}} (A, T), \mathcal{D})$ is called the cohomology of the Rota-Baxter algebra $(A, T)$. We denote the cohomology spaces by $H^\bullet_{RB} (A, T).$

\medskip

\begin{remark}
There is a short exact sequence of cochain complexes
\begin{align*}
\xymatrix{
0 \ar[r] & (C^\bullet (T), \delta_T) \ar[r]^i & (C^\bullet_{\mathrm{RB}} (A, T), \mathcal{D}) \ar[r]^p & (C^\bullet_{\mathrm{Hoch}} (A), \delta_{\mathrm{Hoch}}) \ar[r] & 0,
}
\end{align*}
where $i(P) = (0, P)$ and $p (f,P) = f$. Consequently, there is a long exact sequence in cohomology spaces
\begin{align*}
\xymatrix{
\cdots \ar[r] & H^n(T) \ar[r]^{H^n (i)} & H^n_{\mathrm{RB}} (A, T) \ar[r]^{H^n (p)} & H^n_{\mathrm{Hoch}} (A ) \ar[r]^{\delta^n} & H^{n+1} (T) \ar[r] & \cdots, 
}
\end{align*}
where the connecting homomorphism $\delta^n$ is given by $\delta^n [f] = [\Omega f]$.
\end{remark}

\section{Deformations of relative Rota-Baxter algebras}\label{sec-5}
In this section, we study ${\sf R}$-deformations of a relative Rota-Baxter algebra, where ${\sf R}$ is a local pro-Artinian $\mathbb{K}$-algebra. In particular, we consider the case ${\sf R}=\mathbb{K}[[t]]/(t^2)$, that is the infinitesimal deformations of a relative Rota-Baxter algebra. We show that equivalence classes of infinitesimal deformations of a relative Rota-Baxter algebra $(A,M,T)$ are in bijective correspondence with the cohomology classes in $H^2(A,M,T)$. In the case of a Rota-Baxter algebra $(A,T)$, we can interpret the infinitesimal deformations in terms of the cohomology space $H^2_{RB}(A,T)$ via results similar to this section.

Let ${\sf R}$ be a local pro-Artinian $\mathbb{K}$-algebra. One may define ${\sf R}$-relative Rota-Baxter algebras and morphisms between them similar to the Definition \ref{O-operator} and Definition \ref{morphisms} by replacing the vector spaces and the linear maps by ${\sf R}$-modules and ${\sf R}$-linear maps. Since ${\sf R}$ is local pro-Artinian $\mathbb{K}$-algebra, there is an augmentation $\epsilon: {\sf R}\rightarrow \mathbb{K}$. Thus, any relative Rota-Baxter algebra $(A,M,T)$ can be realised as an ${\sf R}$-relative Rota-Baxter algebra, where the ${\sf R}$-module structures on $A$ and $M$ are respectively given by  
$$r.a=\epsilon(r)a,\quad r.u=\epsilon(r)u, \quad\mbox{for }r\in {\sf R},~ a\in A, \mbox{ and }u\in M.$$

\begin{defn}
 An ${\sf R}$-deformation of a relative Rota-Baxter algebra $((A,\mu),(M,l,r),T)$ consists of an ${\sf R}$-algebra structure $\mu_{\sf R}$ on ${\sf R}\otimes_\mathbb{K} A$, two ${\sf R}$-linear maps $l_{\sf R}:({\sf R}\otimes_\mathbb{K} A)\otimes ({\sf R}\otimes_\mathbb{K} M)\rightarrow ({\sf R}\otimes_\mathbb{K} M)$, $r_{\sf R}:({\sf R}\otimes_\mathbb{K} M)\otimes ({\sf R}\otimes_\mathbb{K} A)\rightarrow ({\sf R}\otimes_\mathbb{K} M)$, and an ${\sf R}$-linear map $T_{\sf R}:{\sf R}\otimes_\mathbb{K} M\rightarrow {\sf R}\otimes_\mathbb{K} A$, which makes the triple $\big(({\sf R}\otimes_\mathbb{K} A,\mu_{\sf R}),({\sf R}\otimes_\mathbb{K} M,l_{\sf R},r_{\sf R}),T_{\sf R}\big)$ into an ${\sf R}$-relative Rota-Baxter algebra such that $$(\epsilon\otimes_\mathbb{K} \mathrm{id}_A,\epsilon\otimes_\mathbb{K} \mathrm{id}_M):\big(({\sf R}\otimes_\mathbb{K} A,\mu_{\sf R}),({\sf R}\otimes_\mathbb{K} M,l_{\sf R},r_{\sf R}),T_{\sf R}\big)\rightarrow ((A,\mu),(M,l,r),T)$$ is a morphism of ${\sf R}$-relative Rota-Baxter algebras. 
\end{defn}

\begin{defn}
Let $\big(({\sf R}\otimes_\mathbb{K} A,\mu_{\sf R}),({\sf R}\otimes_\mathbb{K} M,l_{\sf R},r_{\sf R}),T_{\sf R}\big)$ and $\big(({\sf R}\otimes_\mathbb{K} A,\mu_{\sf R}^{\prime}),({\sf R}\otimes_\mathbb{K} M,l_{\sf R}^{\prime},r_{\sf R}^{\prime}),T_{\sf R}^{\prime}\big)$ be two ${\sf R}$-deformations of a relative Rota-Baxter algebra $((A,\mu),(M,l,r),T)$. Then, they are said to be equivalent if there exists an ${\sf R}$-relative Rota-Baxter algebra isomorphism 
$$(\Phi, \Psi):\big(({\sf R}\otimes_\mathbb{K} A,\mu_{\sf R}),({\sf R}\otimes_\mathbb{K} M,l_{\sf R},r_{\sf R}),T_{\sf R}\big)\rightarrow\big(({\sf R}\otimes_\mathbb{K} A,\mu_{\sf R}^{\prime}),({\sf R}\otimes_\mathbb{K} M,l_{\sf R}^{\prime},r_{\sf R}^{\prime}),T_{\sf R}^{\prime}\big),$$
satisfying $(\epsilon\otimes_\mathbb{K} \mathrm{id}_A)\circ \Phi=(\epsilon\otimes_{\mathbb{K}} \mathrm{id}_A)$ and $(\epsilon\otimes_\mathbb{K} \mathrm{id}_M)\circ \Psi=(\epsilon\otimes_{\mathbb{K}} \mathrm{id}_M)$.
\end{defn}

Next, we explicitly describe ${\sf R}$-deformations of a relative Rota-Baxter algebra for ${\sf R}=\mathbb{K}[[t]]/(t^2)$. Such deformations are called infinitesimal deformations. Note that the augmentation $\epsilon :\mathbb{K}[[t]]/(t^2)\rightarrow \mathbb{K}$ is given by $\epsilon(f)=f(0)$, the evaluation of $f$ at $0$. Thus, an infinitesimal deformation of a relative Rota-Baxter algebra $((A,\mu),(M,l,r),T)$ consists of maps $\mu_0,\mu_1\in \mathrm{Hom}(A\otimes A, A)$, $l_0,l_1\in \mathrm{Hom}(A\otimes M, M)$, $r_0,r_1\in \mathrm{Hom}(M\otimes A, M)$, and $T_0,T_1\in \mathrm{Hom}(M, A)$ such that the maps 
$$\mu_{\sf R}:=\mu_0+t\mu_1,\quad l_{\sf R}:=l_0+t l_1,\quad r_{\sf R}:=r_0+t r_1,\quad T_{\sf R}:T_0+t T_1$$ 
makes $\big(({\sf R}\otimes_\mathbb{K} A,\mu_{\sf R}),({\sf R}\otimes_\mathbb{K} M,l_{\sf R},r_{\sf R}),T_{\sf R}\big)$ into an $\mathbb{K}[[t]]/(t^2)$-relative Rota-Baxter algebra and  
$$(\epsilon\otimes_\mathbb{K} \mathrm{id}_A,\epsilon\otimes_\mathbb{K} \mathrm{id}_M):\big(({\sf R}\otimes_\mathbb{K} A,\mu_{\sf R}),({\sf R}\otimes_\mathbb{K} M,l_{\sf R},r_{\sf R}),T_{\sf R}\big)\rightarrow ((A,\mu),(M,l,r),T)$$
is a morphism of $\mathbb{K}[[t]]/(t^2)$-relative Rota-Baxter algebras. The morphism condition implies that $\mu_0=\mu,~l_0=l,~r_0=r$, and $T_0=T$. Therefore, an infinitesimal deformation is determined by the quadruple $(\mu_1,l_1,r_1,T_1)$.

\begin{thm}\label{infinitesimal-cocycle}
Let an infinitesimal deformation of a relative Rota-Baxter algebra $((A,\mu),(M,l,r),T)$ be determined by the quadruple $(\mu_1,l_1,r_1,T_1)$. Then, the quadruple is a $2$-cocycle in the cohomology of the relative Rota-Baxter algebra $((A,\mu),(M,l,r),T)$. Conversely, any infinitesimal deformation is obtained by a $2$-cocycle in this way.  
\end{thm} 
\begin{proof}
Let ${\sf R}=\mathbb{K}[[t]]/(t^2)$ and $\big(({\sf R}\otimes_\mathbb{K} A,\mu_{\sf R}),({\sf R}\otimes_\mathbb{K} M,l_{\sf R},r_{\sf R}),T_{\sf R}\big)$ be an infinitesimal deformation of relative Rota-Baxter algebra $((A,\mu),(M,l,r),T)$. Then, equivalently we have the following conditions. 

\medskip
\noindent ~~(i) $({\sf R}\otimes_\mathbb{K} A,\mu_{\sf R}=\mu+t\mu_1)$ is an ${\sf R}$-associative algebra, i.e.
$$\mu_{\sf R}(a,\mu_{\sf R}(b,c))=\mu_{\sf R}(\mu_{\sf R}(a,b),c),\quad \mbox{for all }a,b,c \in A.$$
On comparing the coefficients of $t$ from both sides of the above identity, we have
\begin{equation}\label{def-id1}
\mu_1(a,bc)+a\mu_1(b,c)=\mu_1(ab,c)+\mu_1(a,b)c,\quad \mbox{for all }a,b,c \in A,
\end{equation} 

\medskip
\noindent ~~(ii) the triple $({\sf R}\otimes_\mathbb{K} M,l_{\sf R},r_{\sf R})$ is an $({\sf R}\otimes_\mathbb{K} A,\mu_{\sf R}=\mu+t\mu_1)$-bimodule, i.e.
$$l_{\sf R}(\mu_{{\sf R}}(a,b),m)=l_{\sf R}(a,l_{\sf R}(b,m)),\quad r_{\sf R}(r_{\sf R}(m,a),b)=r_{\sf R}(m,\mu_{\sf R}(a,b)),\quad\mbox{and }~~r_{\sf R}(l_{\sf R}(a,m),b)=l_{\sf R}(a,r_{\sf R}(m,b)),$$
for all $a,b\in A$ and $m\in M$.
On comparing the coefficients of $t$ from both sides of the above identities, we have
\begin{equation}\label{def-id2}
\begin{split}
l_1(ab,m)+\mu_1(a,b)m&=l_1(a,bm)+a l_1(b,m)\\
r_1(ma,b)+r_1(m,a)b&=r_1(m,ab)+m\mu_1(a,b)\\
l_1(a,m)b+r_1(am,b)&=l_1(a,mb)+a r_1(m,b),\quad \mbox{for all }a,b \in A, ~~m\in M,
\end{split}
\end{equation} 

\medskip
\noindent ~~(iii) the map $T_{\sf R}:{\sf R}\otimes_\mathbb{K} M\rightarrow {\sf R}\otimes_\mathbb{K} A$ is a relative Rota-Baxter algebra, i.e.
$$\mu_{\sf R}(T_{\sf R}(m),T_{\sf R}(n))=T_{\sf R}\big(l_{\sf R}(T_{\sf R}(m),n)+r_{\sf R}(m,T_{\sf R}(n))\big)$$ 
On comparing the coefficients of $t$ from both sides of the above identity, we have
\begin{equation}\label{def-id3}
\begin{split}
&\mu_1(T(m),T(n))-T\big(l_1(T(m),n)+r_1(m,T(n)\big)\\
=&T_1\big(T(m)n+mT(n)\big)+T\big(T_1(m)n+mT_1(n)\big)-T_1(m)T(n)-T(m)T_1(n),\quad \mbox{for all }m,n\in M.
\end{split}
\end{equation}
By definition of the map $\partial$ is clear that equation \eqref{def-id1} and equation \eqref{def-id2} are equivalent to $\partial(\mu_1+l_1+r_1)=0$. Moreover \eqref{def-id3} is equivalent to $\Omega(\mu_1+l_1+r_1)-d_T(T_1)=0$. Thus,
$\big(({\sf R}\otimes_\mathbb{K} A,\mu_{\sf R}),({\sf R}\otimes_\mathbb{K} M,l_{\sf R},r_{\sf R}),T_{\sf R}\big)$ is an infinitesimal deformation of relative Rota-Baxter algebra $((A,\mu),(M,l,r),T)$ if and only if 
$$\mathcal{D}(\mu_1+l_1+r_1,T_1)=0.$$
Hence, the statement of the theorem follows.
\end{proof}

\begin{thm}
Let $((A,\mu),(M,l,r),T)$ be a relative Rota-Baxter algebra. There is a bijective correspondence between the equivalence classes of infinitesimal deformations of $((A,\mu),(M,l,r),T)$ and the cohomology classes in the second cohomology space $H^2(A,M,T)$.
\end{thm} 
\begin{proof}
Let us fix ${\sf R}=\mathbb{K}[[t]]/(t^2)$. Let $\big(({\sf R}\otimes_\mathbb{K} A,\mu_{\sf R}),({\sf R}\otimes_\mathbb{K} M,l_{\sf R},r_{\sf R}),T_{\sf R}\big)$ and $\big(({\sf R}\otimes_\mathbb{K} A,\mu_{\sf R}^{\prime}),({\sf R}\otimes_\mathbb{K} M,l_{\sf R}^{\prime},r_{\sf R}^{\prime}),T_{\sf R}^{\prime}\big)$ be two infinitesimal deformations of a relative Rota-Baxter algebra $((A,\mu),(M,l,r),T)$. Then, Theorem \ref{infinitesimal-cocycle} implies that these infinitesimal deformations are determined by $2$-cocycles (say) $(\mu_1,l_1,r_1,T_1)$ and $(\mu_1^{\prime},l_1^{\prime},r_1^{\prime},T_1^{\prime})$, respectively. Both of the infinitesimal deformations are equivalent if there exists an ${\sf R}$-relative Rota-Baxter algebra isomorphism 
$$(\Phi, \Psi):\big(({\sf R}\otimes_\mathbb{K} A,\mu_{\sf R}),({\sf R}\otimes_\mathbb{K} M,l_{\sf R},r_{\sf R}),T_{\sf R}\big)\rightarrow\big(({\sf R}\otimes_\mathbb{K} A,\mu_{\sf R}^{\prime}),({\sf R}\otimes_\mathbb{K} M,l_{\sf R}^{\prime},r_{\sf R}^{\prime}),T_{\sf R}^{\prime}\big),$$
satisfying 
\begin{equation}\label{compatibility}
(\epsilon\otimes_\mathbb{K} \mathrm{id}_A)\circ \Phi=(\epsilon\otimes_{\mathbb{K}} \mathrm{id}_A)\quad\mbox{and}\quad(\epsilon\otimes_\mathbb{K} \mathrm{id}_M)\circ \Psi=(\epsilon\otimes_{\mathbb{K}} \mathrm{id}_M).
\end{equation}
From equation \eqref{compatibility}, it follows that 
$$\Phi=\mathrm{id}_A+t \phi_1, \quad \Psi=\mathrm{id}_M+t \psi_1,\quad \mbox{for some } \phi_1\in\mathrm{Hom}(A,A) \mbox{ and } \psi_1\in\mathrm{Hom}(M,M).$$
If $(\Phi, \Psi):\big(({\sf R}\otimes_\mathbb{K} A,\mu_{\sf R}),({\sf R}\otimes_\mathbb{K} M,l_{\sf R},r_{\sf R}),T_{\sf R}\big)\rightarrow\big(({\sf R}\otimes_\mathbb{K} A,\mu_{\sf R}^{\prime}),({\sf R}\otimes_\mathbb{K} M,l_{\sf R}^{\prime},r_{\sf R}^{\prime}),T_{\sf R}^{\prime}\big)$ is an ${\sf R}$-relative Rota-Baxter algebra isomorphism, then one can observe the following.

\medskip
\noindent ~~(i) The map $\mathrm{id}_A+t \phi_1:({\sf R}\otimes_\mathbb{K} A,\mu_{\sf R})\rightarrow ({\sf R}\otimes_\mathbb{K} A,\mu_{\sf R}^{\prime})$ is an associative $\mathsf{R}$-algebra isomorphism, which implies that  
$$\Phi(\mu_{\sf R}(a,b))=\mu_{\sf R}^{\prime}(\Phi(a),\Phi(b)),\quad \mbox{for all }a,b\in A.$$
On comparing the coefficients of $t$,
\begin{equation}\label{eq:id1}
\mu_1-\mu_1^{\prime}=\delta_{\mathrm{Hoch}}\phi_1.
\end{equation}

\medskip
\noindent ~~(ii) The map $(\Phi,\Psi)$ is an isomorphism of $\mathsf{R}$-relative Rota-Baxter operators from $T_{\sf R}$ to $T_{\sf R}^{\prime}$. Then, 
\begin{enumerate}
\item $\Psi(l_{\sf R}(a,m))=l_{\sf R}^{\prime}(\Phi(a),\Psi(m)),~~\Psi(r_{\sf R}(m,a))=r_{\sf R}^{\prime}(\Psi(m),\Phi(a)),~~\mbox{for all }m\in M,~ a\in A$ and 
\item $\Phi\circ T_{\sf R}= T_{\sf R}^{\prime}\circ \Psi$.
\end{enumerate}
Thus, comparing the coefficients of $t$ in the above identities given in the condition $(1)$, we have 
\begin{equation}\label{eq:id2}
\begin{split}
\psi_1(am)+l_1(a,m)&=l_1^{\prime}(a,m)+\phi_1(a)m+a\psi_1(m)\\
\psi_1(ma)+r_1(m,a)&=r_1^{\prime}(m,a)+\psi_1(m)a+m\phi_1(a).
\end{split}
\end{equation}
 Moreover, comparing the coefficients of $t$ from both sides of the identity in the condition $(2)$, we get 
 \begin{equation}\label{eq:id3}
 T_1+\phi_1\circ T= T_1^{\prime}+T\circ \psi_1. 
\end{equation}  

Therefore, from equations \eqref{eq:id1}-\eqref{eq:id3}, the pair $(\Phi,\Psi)$ is an ${\sf R}$-relative Rota-Baxter algebra isomorphism if and only if 
$$(\mu_1,l_1,r_1,T_1)-(\mu_1^{\prime},l_1^{\prime},r_1^{\prime},T_1)=\mathcal{D}(\phi_1,\psi_1).$$
Hence, there is a bijective correspondence between the equivalence classes of infinitesimal deformations of $((A,\mu),(M,l,r),T)$ and the second cohomology space $H^2(A,M,T)$. 
\end{proof}

\section{Cohomology and deformations of skew-symmetric infinitesimal bialgebras}\label{sec-6}
In this section, we define cohomology of (coboundary) skew-symmetric infinitesimal bialgebras coming from a skew-symmetric associative $r$-matrix. Throughout the section, we assume that all the vector spaces are finite-dimensional.
\subsection{Skew-symmetric $r$-matrices and skew-symmetric infinitesimal bialgebras}
Let $A$ be an associative algebra and let $r\in \wedge^2 A$ given by $r:=\sum_i a_i\otimes b_i\in A\otimes A$. We recall from \cite{bai1} that the associative Yang-Baxter equation in $A$ is given by 
\begin{equation}\label{Associative YB eq}
r_{12}r_{13}+r_{13}r_{23}-r_{23}r_{12}=0,
\end{equation}
where 
\begin{equation}
\begin{split}
r_{12}r_{13}&=\sum_{ij}a_i a_j \otimes b_i\otimes b_j,\\
r_{13}r_{23}&=\sum_{ij}a_i  \otimes a_j\otimes b_i b_j,\\
r_{23}r_{12}&=\sum_{ij} a_j \otimes a_i b_j\otimes b_i.
\end{split}
\end{equation}
A solution $r\in \wedge^2 A$ of the associative Yang-Baxter equation \eqref{Associative YB eq} is called a skew-symmetric $r$-matrix on associative algebra $A$. An element $r\in \wedge^2 A$ induces a linear map $r^\sharp A^*\rightarrow A$ given by 
$$\langle\beta,r^\sharp(\alpha)\rangle:=r(\alpha,\beta), \quad\mbox{for all } \alpha,\beta\in A^*.$$ 
Let us also recall the following relationship between the skew-symmetric $r$-matrices and relative Rota-Baxter operators on associative algebras, see \cite{bai1} for details.

\begin{prop}
An element $r\in \wedge^2 A$ is a skew-symmetric associative $r$-matrix on an associative algebra $A$ if and only if $(A,A^*,r^\sharp)$ is a relative Rota-Baxter algebra. Here, $A^*$ is an $A$-bimodule with coadjoint actions.
\end{prop}

Let $L_a:A\rightarrow A$ and $R_a:A\rightarrow A$ be left and right multiplication operators on the associative algebra $A$ for $a\in A$. 

\begin{defn}[\cite{bai1}]
A skew-symmetric (antisymmetric) infinitesimal bialgebra structure on an associative algebra $A$ is a linear map $\Delta:A\rightarrow A\otimes A$ such that the map $\Delta^*:A^*\otimes A^*\rightarrow A^*$ defines an associative algebra structure on $A^*$, and for all $a,b\in A$, the map $\Delta$ satisfies the following equations
\begin{equation}\label{Aib1}
\Delta(ab)=(\mathrm{id}\otimes L_a)\Delta(y)+(R_b\otimes \mathrm{id})\Delta(x),
\end{equation}
\begin{equation}\label{Aib2}
(L_b\otimes\mathrm{id}-\mathrm{id}\otimes R_b)\Delta(a)+\sigma\big((L_a\otimes\mathrm{id}-\mathrm{id}\otimes R_a)\Delta(b)]=0.
\end{equation}
Here, the map $\sigma : A\otimes A\rightarrow A\otimes A$ is defined by $\sigma(a\otimes b):=b\otimes a$.
\end{defn}
A skew-symmetric infinitesimal bialgebra $(A,\Delta)$ is called \textit{coboundary skew-symmetric infinitesimal bialgebra} if there exists an element $r\in A\otimes A$ such that 
\begin{equation}\label{coboundaryAIB}
\Delta_r(a)=(\mathrm{id}\otimes L_a-R_a\otimes \mathrm{id}) r, \quad\mbox{for all }a\in A.
\end{equation}

\begin{prop}[\cite{bai1}]
Let $A$ be an associative algebra and $r\in \wedge^2 A$. Then, the map $\Delta_r$ given by the equation \eqref{coboundaryAIB} defines an skew-symmetric infinitesimal bialgebra $(A,\Delta_r)$ if $r$ satisfies the associative Yang-Baxter equation \eqref{Associative YB eq} in $A$. 
\end{prop}

Let $(A_1,\Delta_{r_1})$ and $(A_2,\Delta_{r_2})$ be two coboundary skew-symmetric infinitesimal bialgebras. A homomorphism of  coboundary skew-symmetric infinitesimal bialgebras is an algebra homomorphism $\phi:A_1\rightarrow A_2 $ satisfying $(\phi\otimes \phi)(r_1)=r_2$.
\subsection{Cohomology of coboundary skew-symmetric infinitesimal bialgebras}
Let $A$ be an associative algebra and $r\in \wedge^2 A$ be a solution of associative Yang-Baxter equation \eqref{Associative YB eq}. Then, $(A,\Delta_r)$ is a coboundary skew-symmetric infinitesimal bialgebras.

The skew-symmetric matrix $r$ corresponds to a relative Rota-Baxter operator $r^\sharp: A^*\rightarrow A$. Thus, we have an associative algebra structure on $A^*$ given by 
$$\alpha\star \beta =r^\sharp(\alpha)\beta + \alpha r^\sharp(\beta ), \quad\mbox{for all }\alpha, \beta \in A^*.$$
Here, the first and second term on the right hand side denotes left and right coadjoint actions of $A$ on $A^*$, respectively . Let us consider the cyclic complex $(C_\lambda^\bullet(A^*),b)$ for the associative algebra $A^*$. We define a cochain complex $(C^\bullet_r(A),\delta_r)$ associated to an skew-symmetric $r$-matrix as follows: for $n=0,1$, define $C^n_r(A):=0$ and for $n\geq 2$, define $C^n_r(A):=C^{n-1}_\lambda(A^*)$. The coboundary map $\delta_r:= b$. We denote the cohomology of the complex by $H^\bullet_r(A)$.

Let us consider the canonical monomorphism $\varphi: C^\bullet_\lambda(A^*)\rightarrow C^\bullet_{Hoch}(A^*,A)$ given by 
$$\langle\varphi(\eta)(\alpha_1,\ldots,\alpha_n),\alpha_{n+1}\rangle:=\eta(\alpha_1,\ldots,\alpha_n,\alpha_{n+1}), \quad\mbox{for all } \alpha_1,\ldots,\alpha_{n+1}\in A^*.$$
Then, the following diagram commutes 
\begin{equation}\label{subcomplex}
\begin{CD}
C^{n-1}_\lambda(A^*)@>\varphi>>C^{n-1}_{Hoch}(A^*,A)\\
  @V \delta_r VV @VV\delta_{Hoch} V\\
C^{n}_\lambda(A^*)@>\varphi>>C^{n}_{Hoch}(A^*,A)
\end{CD}
\end{equation}

\noindent In fact, for any $\eta\in C^{n-1}_{\lambda}(A^*)$ and $\alpha_1,\ldots,\alpha_n\in A^*$, we get 
\begin{equation}\label{com diag id1}
\begin{split}
\delta_{Hoch}(\varphi(\eta))(\alpha_1,\alpha_2,\ldots,\alpha_n)
=&(-1)^{n-1}r^{\sharp}(\varphi(\eta)(\alpha_1,\ldots,\alpha_{n-1})\alpha_n)-r^{\sharp}(\alpha_1\varphi(\eta)(\alpha_2,\ldots,\alpha_n))\\
&+\sum_{i=1}^{n-1}(-1)^{i-1}\varphi(\eta)(\alpha_1,\ldots,\alpha_i\star \alpha_{i+1},\ldots,\alpha_n)\\
&+r^{\sharp}(\alpha_1)\varphi(\eta)(\alpha_2,\ldots,\alpha_{n})-(-1)^{n-1}\varphi(\eta)(\alpha_1,\ldots,\alpha_{n-1})r^{\sharp}(\alpha_n).
\end{split}
\end{equation}
For any $\beta\in A^*$, the first and second term on the right hand side of the equation \eqref{com diag id1} can be written as
\begin{equation}\label{com diag id2}
\begin{split}
\langle r^{\sharp}(\varphi(\eta)(\alpha_1,\ldots,\alpha_{n-1})\alpha_n),\beta \rangle=-\langle \varphi(\eta)(\alpha_1,\ldots,\alpha_{n-1})\alpha_n, r^{\sharp}(\beta )\rangle&=-\langle \varphi(\eta)(\alpha_1,\ldots,\alpha_{n-1}), \alpha_n r^{\sharp}(\beta )\rangle\\
&=-\eta(\alpha_1,\ldots,\alpha_{n-1},\alpha_n r^{\sharp}(\beta )),
\end{split}
\end{equation}
\begin{equation}\label{com diag id3}
\langle r^{\sharp}(\alpha_1\varphi(\eta)(\alpha_2,\ldots,\alpha_{n})),\beta \rangle=-\langle \varphi(\eta)(\alpha_2,\ldots,\alpha_{n})),r^{\sharp}(\beta)\alpha_1 \rangle=-\eta(\alpha_2,\ldots,\alpha_{n},r^\sharp(\beta)\alpha_1).
\end{equation}
Similarly, for any $\beta\in A^*$, the fourth and fifth term on the right hand side of the equation \eqref{com diag id1} can be written as 
\begin{equation}\label{com diag id4}
\begin{split}
\langle r^{\sharp}(\alpha_1)\varphi(\eta)(\alpha_2,\ldots,\alpha_{n}),\beta \rangle=\langle \varphi(\eta)(\alpha_2,\ldots,\alpha_{n}), \beta r^{\sharp}(\alpha_1)\rangle
=\eta(\alpha_2,\ldots,\alpha_{n},\beta r^{\sharp}(\alpha_1)),
\end{split}
\end{equation}
\begin{equation}\label{com diag id5}
\langle \varphi(\eta)(\alpha_1,\ldots,\alpha_{n-1})r^{\sharp}(\alpha_n),\beta\rangle=\eta(\alpha_1,\alpha_2,\ldots,\alpha_{n-1},r^\sharp(\alpha_n)\beta).
\end{equation}

On the other hand, by the definition of the cyclic coboundary $\delta_r=b:C^{n-1}_\lambda(A^*)\rightarrow C^{n}_\lambda(A^*)$, for $\eta\in C^{n-1}_\lambda(A^*)$, we have
\begin{equation}\label{com diag id6}
\begin{split}
\langle \varphi(b(\eta))(\alpha_1,\ldots,\alpha_{n-1},\alpha_n),\beta\rangle&=b(\eta)(\alpha_1,\ldots,\alpha_{n-1},\alpha_n,\beta)\\
&=\sum_{i=1}^{n-1}(-1)^{i+1}\eta(\alpha_1,\ldots,\alpha_i\star\alpha_{i+1},\ldots,\alpha_n,\beta)\\&\quad+(-1)^{n+1}\eta(\alpha_1,\ldots,\alpha_{n-1},\alpha_n\star\beta)+(-1)^{n}\eta(\beta\star\alpha_1,\alpha_2\ldots,\alpha_n).
\end{split}
\end{equation}
 From equations \eqref{com diag id1}-\eqref{com diag id6}, it follows that for any $\eta\in C^{n-1}_\lambda(A^*)$, 
 $$\langle \delta_{Hoch}(\varphi(\eta))(\alpha_1,\ldots,\alpha_n),\beta\rangle=\langle \varphi(b(\eta))(\alpha_1,\ldots,\alpha_{n-1},\alpha_n),\beta\rangle, \quad\mbox{for all }\alpha_1,\ldots,\alpha_n,\beta\in A^*.$$
Therefore, the diagram \eqref{subcomplex} commutes. The complex $(C^{\bullet-1}_{Hoch}(A^*,A),(-1)^\bullet\delta_{Hoch})$ is the same as the complex $(C^\bullet(r^\sharp),d_{r^\sharp})$ associated to the relative Rota-Baxter operator $r^\sharp: A^*\rightarrow A$. Thus, the diagram \eqref{subcomplex} implies that $(Im\varphi,d_{r^\sharp})$ is a subcomplex of the cochain complex $(C^\bullet(r^\sharp),d_{r^\sharp})$.

Let us now define a complex $(C^\bullet(A,\Delta_r),\delta_{inf})$ for the coboundary infinitesimal bialgebra $(A,\Delta_r)$ with $r\in \wedge^2 A$. 
\begin{itemize}
\item For $n=0$, define $C^0(A,\Delta_r):=0$;
\item For $n=1$, define $C^1(A,\Delta_r):=Hom(A,A)$;
\item For $n\geq 2$, define $C^n(A,\Delta_r):=C^n_{Hoch}(A)\oplus C^{n-1}_\lambda(A^*)$.   
\end{itemize}
To define the coboundary map $\delta_{inf}$, we first define an embedding $\mathfrak{i}:C^\bullet(A,\Delta_r)\rightarrow C^\bullet(A,A^*,r^\sharp)$. For $n=1$, the map $\mathfrak{i}:C^1(A,\Delta_r)=\mathrm{Hom}(A,A)\rightarrow C^1(A,A^*,r^\sharp)=\mathrm{Hom}(A,A)\oplus \mathrm{Hom}(A^*,A^*)$ is given by
$i(f)=(f,\bar{f})$, for $f\in \mathrm{Hom}(A,A)$. Here, the map $\bar{f}\in \mathrm{Hom}(A^*,A^*)$ is the dual of $f$.
\smallskip

\noindent For $n\geq 2$, let us recall that $C^n(A,A^*,r^\sharp)=\mathrm{Hom}(\otimes^n A,A)\oplus \mathrm{Hom}(\mathcal{A}^{n-1,1},A^*)\oplus\mathrm{Hom}(\otimes^{n-1}A^*,A),$ where $\mathcal{A}^{n-1,1}$ are all those $n$-tensor powers of $A$ and $A^*$ such that $A^*$ appears only once. Define the map $\mathfrak{i}:C^n(A,\Delta_r)\rightarrow C^n(A,A^*,r^\sharp)$ by 
$$\mathfrak{i}(f,\eta)=(f,\bar{f},\varphi(\eta))$$
for $f\in \mathrm{Hom}(\otimes^n A,A)$ and $ \eta\in C_\lambda^{n-1}(A^*)$. Here, $\bar{f}\in \mathrm{Hom}(\mathcal{A}^{n-1,1},A^*)$ is given by
$$\langle\bar{f}(a_1,\ldots,\underbrace{\alpha}_{i^{th}-\text{position}},\ldots,a_{n}),b\rangle:=(-1)^{n-i+1}\langle\alpha,f(a_1,\ldots,\underbrace{b}_{i^{th}-\text{position}},\ldots,a_n)\rangle,$$
for any $\alpha\in A^*$ and $a_1,\ldots,\widehat{a_i},\ldots,a_n,b\in A$.
Let $f\in \mathrm{Hom}(\otimes^n A,A)$, then observe that $\Omega(f+\bar{f})$ is skew-symmetric and thus by the definition of the coboundary $\mathcal{D}$, we have 
$$\mathcal{D}(\mathfrak{i} (f))=\mathcal{D}(f,\bar{f})=\partial(f+\bar{f}),\Omega(f+\bar{f})=\mathfrak{i}(\delta_{Hoch}(f),\varphi^{-1}(\Omega(f+\bar{f})) ).$$

For $n\geq 2$ and $(f,\eta)\in C^n(A,\Delta_r)$, 
$$ \mathcal{D}(\mathfrak{i} (f, \eta)) =\mathcal{D}(f,\bar{f}, \varphi(\eta))=( \partial (f+\bar{f}), (-1)^{n} d_{r^\sharp}\varphi(\eta) + \Omega (f+\bar{f})),$$
Since $(-1)^{n} d_{r^\sharp} \varphi(\eta) + \Omega (f+\bar{f})\in Im\varphi$, we get
$$\mathcal{D}(\mathfrak{i} (f, \eta))=\mathfrak{i}(\delta_{Hoch}f,\varphi^{-1}((-1)^{n} d_{r^\sharp} \varphi(\eta) + \Omega (f+\bar{f})))$$
Thus, $\mathcal{D}(\mathfrak{i}(f, \eta))\in Im(\mathfrak{i})$ for any $(f,\eta)\in C^\bullet(A,\Delta_r)$. Let us consider a map $\theta:Im(\mathfrak{i})\rightarrow C^\bullet(A,\Delta_r)$ given by
$\theta(\mathfrak{i}(f,\eta))=(f,\eta)$. Now, define the coboundary $\delta_{inf}:C^{\bullet}(A,\Delta_r)\rightarrow C^{\bullet+1}(A,\Delta_r)$ by the composition $\delta_{inf}:=\theta\circ \mathcal{D}\circ\mathfrak{i}$. The condition $\mathcal{D}(Im(\mathfrak{i}))\subseteq Im(\mathfrak{i})$ and the fact that $\mathfrak{i}\circ \theta|_{Im(\mathfrak{i})}=\mathrm{id}$ imply that $\delta_{inf}\circ \delta_{inf}=0$. Let us denote the cohomology of the cochain complex $(C^\bullet (A,\Delta_r),\delta_{inf})$ by $H^\bullet_{inf}(A,\Delta_r)$. We call it the \textit{cohomology of the coboundary skew-symmetric infinitesimal bialgebra $(A,\Delta_r)$}. 

Let us observe that by the definition of the coboundary map $\delta_{inf}$, it can be rewritten as 
$$\delta_{inf}(f,\eta)=(\delta_{Hoch}(f),\delta_r(\eta)+\varphi^{-1}(\Omega(f+\bar{f}))),$$
for any $(f,\eta)\in C^n(A,\Delta_r)$. Thus, we have a short exact sequence of complexes 
\begin{align*}
\xymatrix{
0 \ar[r] & (\oplus_{n\geq 2}C^{n-1}_\lambda (A^*), b) \ar[r]^{i} & (\oplus_{n\geq 1} C^n(A, \Delta_r), \delta_{inf}) \ar[r]^p & (\oplus_{n\geq 1} C^n_{\mathrm{Hoch}} (A), \delta_{\mathrm{Hoch}}) \ar[r] & 0,
}
\end{align*}
where $i(\eta) = (0, \eta)$ and $p (f,\eta) = f$ for any $f\in C^\bullet_{Hoch}(A)$ and $\eta\in C^\bullet_r (A)=C^{\bullet-1}_\lambda(A^*)$. Therefore, we have the following result.

\begin{thm}
There is a long exact sequence of cohomology spaces
\begin{align*}
\xymatrix{
\cdots \ar[r] & HC^{n-1}(A^*) \ar[r]^{H^n (i)} & H^n_{inf} (A, \Delta_r) \ar[r]^{H^n (p)} & H^n_{\mathrm{Hoch}} (A) \ar[r]^{\delta^n} & HC^{n} (A^*) \ar[r] & \cdots, 
}
\end{align*}
where the connecting homomorphism $\delta^n$ is given by $\delta^n [f] = [\varphi^{-1}(\Omega (f+\bar{f}))]$, for any $[f]\in H^n_{Hoch}(A)$.
\qed
\end{thm}

\subsection{Deformations of coboundary skew-symmetric infinitesimal bialgebras}
Let ${\sf R}$ be a local pro-Artinian $\mathbb{K}$-algebra and thus, equipped with an augmentation $\epsilon: {\sf R}\rightarrow \mathbb{K}$. Let $A$ be an associative algebra and $r\in \wedge^2 A$. Then, any coboundary skew-symmetric infinitesimal bialgebra $(A,\Delta_r)$ can be realised as a skew-symmetric ${\sf R}$-infinitesimal bialgebra, where the ${\sf R}$-module structures on $A$ is given by  
$$r.a=\epsilon(r)a, \quad\mbox{for any }r\in {\sf R},~ a\in A.$$

\begin{defn}\label{R-def of SIB}
Let $A$ be an associative algebra and $r\in \wedge^2 A$ be an associative $r$-matrix. An ${\sf R}$-deformation of a coboundary skew-symmetric infinitesimal bialgebra $(A,\Delta_r)$ is a triplet $(R\otimes A,\mu_R,r_R)$ consisting of an ${\sf R}$-linear associative product $\mu_{\sf R}:({\sf R}\otimes A)\otimes_R({\sf R}\otimes A)\rightarrow ({\sf R}\otimes A)$ and a skew-symmetric $r$-matrix $r_{\sf R}\in ({\sf R}\otimes A)\otimes_R ({\sf R}\otimes A)$ such that 
\begin{enumerate}[(i)]
\item the map $\epsilon \otimes \mathrm{id}:({\sf R}\otimes A,\mu_{\sf R}) \rightarrow (A,\mu)$ is an ${\sf R}$-associative algebra homomorphism, and 
\item $(\epsilon\otimes \mathrm{id})\otimes_R(\epsilon\otimes \mathrm{id})(r_R)=r$.
\end{enumerate}

\end{defn}

Let $(A,\Delta_r)$ be a coboundary skew-symmetric infinitesimal bialgebra. Two ${\sf R}$-deformations of $(A,\Delta_r)$ given by the triplets $(R\otimes A,\mu_R,r_R)$ and $(R\otimes A,\mu_R^\prime,r_R^\prime)$ are said to be \textit{equivalent} if there exists an isomorphism of coboundary skew-symmetric ${\sf R}$-infinitesimal bialgebras $$\Phi: (R\otimes A,\mu_R,r_R)\rightarrow (R\otimes A,\mu_R^\prime,r_R^\prime)\quad \text{ satisfying }\quad (\epsilon\otimes \mathrm{id})\circ \Phi=\epsilon\otimes \mathrm{id}.$$  

Let ${\sf R}=\mathbb{K}[t]/(t^2)$, then an  ${\sf R}$-deformation $(R\otimes A,\mu_R,r_R)$ is called an \textit{infinitesimal deformation} of $(A,\Delta_r)$. There exists $\mu_0,\mu_1\in C^2_{Hoch}(A)$ and $r_0,r_1\in C^1_\lambda(A^*)$ such that 
$\mu_R=\mu_0+t \mu_1$ and $r_R=r_0+t r_1.$ 
The conditions (i) and (ii) of Definition \ref{R-def of SIB} imply that $\mu_0=\mu$ and $r_0=r$.

\begin{thm}
Let $(A,\Delta_r)$ be a coboundary skew-symmetric infinitesimal bialgebra. Let $\mu_1\in C^2_{Hoch}(A)$ and $r_1\in C^1_\lambda(A^*)$. Then an infinitesimal deformation of $(A,\Delta_r)$ is determined by the pair $(\mu_1,r_1)$ if and only if $(\mu_1,r_1) \in C^2(A,\Delta_r)$ is a $2$-cocycle.
\end{thm}

\begin{proof}
The triplet $(R\otimes A,\mu_R:=\mu+t \mu_1,r_R:=r+ t r_1)$ is an ${\sf R}$-deformation if and only if 
\begin{enumerate}[(i)]
\item $(R\otimes A,\mu_R:=\mu+t \mu_1)$ is an $R$-associative algebra, or equivalently $\delta_{Hoch}(\mu_1)=0$, and
\item $r_R:=r+ t r_1\in ({\sf R}\otimes A)\otimes_R ({\sf R}\otimes A)$ is a skew-symmetric $r$-matrix of the $R$-associative algebra $(R\otimes A,\mu_R:=\mu+t \mu_1)$, or equivalently $\delta_r(r_1)+\varphi^{-1}(\Omega(\mu_1+\bar{\mu_1}))=0$.
\end{enumerate} 
Therefore, the triplet $(R\otimes A,\mu_R:=\mu+t \mu_1,r_R:=r+ t r_1)$ is an ${\sf R}$-deformation if and only if
$$\delta_{inf}(\mu_1,r_1)=(\delta_{Hoch}(\mu_1),\delta_r(r_1)+\varphi^{-1}(\Omega(\mu_1+\bar{\mu_1})))=0.$$
Hence, the theorem follows.
\end{proof}

\begin{thm}
The equivalence classes of infinitesimal deformations of $(A,\Delta_r)$ corresponds bijectively to cohomology classes in $H^2(A,\Delta_r)$.
\end{thm}

\begin{proof}
Two infinitesimal deformations $(R\otimes A,\mu_R:=\mu+t \mu_1,r_R:=r+ t r_1)$ and $(R\otimes A,\mu_R^\prime:=\mu+t \mu_1^\prime,r_R^\prime:=r+ t r_1^\prime)$ of $(A,\Delta_r)$ are equivalent if and only if there exists an isomorphism of coboundary skew-symmetric ${\sf R}$-infinitesimal bialgebras $$\Phi: (R\otimes A,\mu_R,r_R)\rightarrow (R\otimes A,\mu_R^\prime,r_R^\prime)\quad \text{ satisfying }\quad (\epsilon\otimes \mathrm{id})\circ \Phi=\epsilon\otimes \mathrm{id}.$$
The condition $(\epsilon\otimes \mathrm{id})\circ \Phi=\epsilon\otimes \mathrm{id}$ implies that there exists an element $\phi_1\in C^1_{Hoch}(A)$ such that $\Phi=\mathrm{id}+t \phi_1$. The map $\Phi: (R\otimes A,\mu_R,r_R)\rightarrow (R\otimes A,\mu_R^\prime,r_R^\prime)$ is an isomorphism if and only if 
\begin{enumerate}[(i)]
\item The map $\mathrm{id}+t \phi_1:(R\otimes A,\mu_R:=\mu+t \mu_1)\rightarrow (R\otimes A,\mu_R^\prime:=\mu+t \mu_1^\prime)$ is an isomorphism of $R$-associative algebras, or equivalently $\mu_1-\mu_1^\prime=\delta_{Hoch}(\phi_1),$ and
\item $(\mathrm{id}+t \phi_1)\otimes_R(\mathrm{id}+t \phi_1)(r+t r_1)=r+t r_1^\prime$ or equivalently $r_1-r_1^\prime=\varphi^{-1}(\Omega(\phi_1+\bar{\phi_1}))$.
\end{enumerate}
The pairs $(\mu_1,r_1)$ and $(\mu_1^\prime,r_1^\prime)$ are $2$-cocycles in $C^2(A,\Delta_r)$. By $(i)$ and $(ii)$ it is clear that $\Phi=\mathrm{id}+t \phi_1: (R\otimes A,\mu_R,r_R)\rightarrow (R\otimes A,\mu_R^\prime,r_R^\prime)$ is an isomorphism if and only if $(\mu_1,r_1)-(\mu_1^\prime,r_1^\prime)=\delta_{inf}(\phi_1)$. Thus, $(R\otimes A,\mu_R,r_R)$ and $(R\otimes A,\mu_R^\prime,r_R^\prime)$ are equivalent infinitesimal deformations if and only if $(\mu_1,r_1)$ and $(\mu_1^\prime,r_1^\prime)$ belong to the same cohomology class in $H^2(A,\Delta_r)$.
\end{proof}
\section{Homotopy relative Rota-Baxter operators}\label{sec-7}
In this section, we introduce homotopy relative Rota-Baxter operators on bimodules over strongly homotopy associative algebras. Our notion will generalize (strict) Rota-Baxter operators on $A_\infty$-algebras introduced in \cite{das1}. We compare homotopy relative Rota-Baxter algebras and homotopy relative Rota-Baxter Lie algebras introduced in \cite{laza-rota}. Finally, we construct a $pre\text{-}Lie_\infty [1]$-algebra from any $Dend_\infty [1]$-algebra generalizing the similar result from non-homotopic case \cite{aguiar-pre}. This construction suitably fits with the various relations among dendriform algebras, pre-Lie algebras, associative algebras, Lie algebras and relative Rota-Baxter algebras in the homotopy context. 

\begin{defn}
An $A_\infty [1]$-algebra is a graded vector space $A = \bigoplus_{i \in \mathbb{Z}} A_i$ together with a collection of degree $1$ multilinear maps $\{ \mu_k : A^{\otimes k} \rightarrow A \}_{k \geq 1}$ satisfying the following identities
\begin{align}\label{a-inf-iden}
\sum_{i+j = n+1}  \sum_{\lambda = 1}^j (-1)^{|a_1|+ \cdots + |a_{\lambda -1}|} ~\mu_j ( a_1, \ldots, a_{\lambda -1} , \mu_i ( a_\lambda, \ldots, a_{\lambda + i -1}), a_{\lambda +i}, \ldots, a_{n} ) = 0, ~ \text{ for } n \geq 1.
\end{align}
\end{defn}

Note that $A_\infty [1]$-algebra structure on a graded vector space $A$ can be described by Maurer-Cartan element in a suitable graded Lie algebra \cite{stasheff}. Let $A$ be a graded vector space. Consider the free reduced tensor algebra $\overline{T} (A) = \bigoplus_{n \geq 1} A^{\otimes n}$ and for each $n \in \mathbb{Z}$, define $C^n(A) = \mathrm{Hom}_n (\overline{T}(A), A)$ the space of degree $n$ linear maps. An element $\mu \in \mathrm{Hom}_n (\overline{T}(A), A)$ is the sum of degree $n$  multilinear maps $\mu_k : A^{\otimes k} \rightarrow A$, for $ k \geq 1$.

With these notations $\bigoplus_{n \in \mathbb{Z}} C^n (A) = \bigoplus_{n \in \mathbb{Z}} \mathrm{Hom}_n (T(A), A)$ is a graded Lie algebra with the bracket given by $[\mu_k, \gamma_l ] = \mu_k \circ \gamma_l - (-1)^{mn} \gamma_l \circ \mu_k$, where
\begin{align*}
(\mu_k \circ \gamma_l ) (a_1, \ldots, a_{k+l-1}) = \sum_{i = 1}^k (-1)^{  |a_1|+ \cdots + |a_{i -1}|  }     \mu_k  ( a_1, \ldots, a_{i -1}, \gamma_l ( a_i , \ldots, a_{i + l -1}), a_{i + l}, \ldots, a_{k+l-1} ),
\end{align*}
for $\mu = \sum_{k \geq 1} \mu_k \in C^m (A)$ and $\gamma = \sum_{l \geq 1} \gamma_l \in C^n(A).$ An element $\mu \in C^1(A) = \sum_{k \geq 1} \mu_k$ is a Maurer-Cartan element in the above graded Lie algebra if and only if $(A, \{ \mu_k \}_{k \geq 1})$ is an $A_\infty [1]$-algebra.

\begin{defn}
Let $(A, \{ \mu_k \}_{k \geq 1})$ be an $A_\infty [1]$-algebra. A bimodule over it consists of a graded vector space $M = \bigoplus_{i \in \mathbb{Z}} M_i$ together with a collection of degree $1$ maps $\{ \eta_k : \mathcal{A}^{k-1, 1} \rightarrow M \}_{k \geq 1}$ satisfying the identities (\ref{a-inf-iden}) with exactly one of $a_1, \ldots, a_n$ is from $M$ and the corresponding multilinear operation $\mu_i$ or $\mu_j$ replaced by $\eta_i$ or $\eta_j$. Here, $\mathcal{A}^{k-1,1}$ denotes all those $k$-tensor powers of $A$ and $M$, in which $M$ appears only once.
\end{defn} 

It follows from the definition that any $A_\infty[1]$-algebra is a bimodule over itself with $\eta_k = \mu_k$, for $k \geq 1$. This is called the adjoint bimodule.

\subsection{Homotopy relative Rota-Baxter operators}
Here, we use a homotopy version of Theorem \ref{das-gla} to describe relative Rota-Baxter operators in the homotopy context. 

Let $(A, \{\mu_k \}_{k \geq 1})$ be an $A_\infty [1]$-algebra and $(M, \{ \eta_k \}_{k \geq 1})$ be a bimodule over it. We consider the graded Lie algebra
\begin{align*}
L = \big(   \bigoplus_{n \in \mathbb{Z}} C^n (A \oplus M) = \bigoplus_{ n \in \mathbb{Z}} \mathrm{Hom}_n ( \overline{T}(A \oplus M), A \oplus M), ~ [~,~] \big).
\end{align*}
It is easy to see that $\mathfrak{a} = \bigoplus_{n \in \mathbb{Z}} \mathrm{Hom}_n ( \overline{T}(M), A)$ is an abelian subalgebra of $L$. Let $P : L \rightarrow \mathfrak{a}$ be the projection onto the subspace $\mathfrak{a}$. Then $\triangle = \sum_{k \geq 1} (\mu_k + \eta_k) \in \mathrm{ker}(P)_1$ and satisfies $[ \triangle, \triangle ] = 0$. Therefore, we get that $(L, \mathfrak{a}, P, \triangle )$ is a $V$-data. Hence by Theorem \ref{Voro-JPAA1}, we get an $L_\infty [1]$-algebra generalizing the graded Lie algebra of Theorem \ref{das-gla} to the homotopy context.

\begin{thm}
Let $(A, \{\mu_k \}_{k \geq 1})$ be an $A_\infty [1]$-algebra and $(M, \{ \gamma_k \}_{k \geq 1})$ be a bimodule over it. Then there is an $L_\infty [1]$-algebra structure on the graded vector space $\mathfrak{a} = \bigoplus_{n \in \mathbb{Z}} \mathrm{Hom}_n ( T(M), A)$. Moreover, this $L_\infty[1]$-algebra is weakly filtered. \qed
\end{thm}

Note that the filtration here is given by $\mathcal{F}_n ( \mathfrak{a}) = \bigoplus_{i \geq n} \mathrm{Hom} ( M^{\otimes i} , A )$, for $n \geq 1$.
Inspired by the characterization of relative Rota-Baxter operators in the usual case, we define the following definition.

\begin{defn}
Let $(A, \{ \mu_k \}_{k \geq 1})$ be an $A_\infty [1]$-algebra and $(M, \{\gamma_k \}_{k \geq 1})$ be a bimodule. A Maurer-Cartan element $T =\sum_{k \geq 1} T_k \in \mathrm{Hom}_0 ( T(M), A)$ in the above $L_\infty [1]$-algebra is called a homotopy relative Rota-Baxter operator on $(M, \{ \gamma_k \}_{k \geq 1})$ over the $A_\infty [1]$-algebra $(A, \{ \mu_k \}_{k \geq 1})$.
\end{defn}

In the next, we give a characterization of a homotopy relative Rota-Baxter operator $T = \sum_{k \geq 1} T_k$ in terms of its components.

\begin{prop}
An element $T = \sum_{k \geq 1} T_k \in \mathrm{Hom}_0 ( T(M), A) $ is a homotopy relative Rota-Baxter operator if and only if its components satisfy the following identities: for each $p \geq 1$,
\begin{align}\label{homo-rota-iden}
&\sum_{k_1 + \cdots + k_n = p} \frac{1}{n!}~\mu_n \big( T_{k_1} (u_1, \ldots, u_{k_1}), T_{k_2} (u_{k_1 + 1}, \ldots, u_{k_1 + k_2}), \ldots, T_{k_n } (u_{k_1 + \cdots + k_{n-1} +1} , \ldots, u_{p})   \big) \\
&=  \sum_{i=1}^n \sum_{ \substack{j+k_1+ \cdots + k_{i-1} +1 \\ + k_{i+1} + \cdots + k_n + t = p}} \frac{(-1)^{|u_1| + \cdots + |u_j|}}{(n-1)!}  ~ T_{j+1+t} \bigg(   u_1, \ldots, u_j , \eta \big(   T_{k_1} ( u_{j+1} , \ldots u_{j+ k_1} ), \ldots, 
T_{k_{i-1}} ( \cdots , u_{j+ k_1 + \cdots + k_{i-1}} ), \nonumber\\ 
&  \qquad \qquad \qquad \qquad u_{j + k_1 + \cdots + k_{i-1} +1}, T_{k_{i+1}} (\ldots, u_{j+k_1+ \cdots + k_{i-1} +1 + k_{i+1}}), \ldots, T_{k_n} (\ldots, u_{p-t}) \big), u_{p-t+1}, \ldots, u_p \bigg). \nonumber
\end{align}
\end{prop}

\begin{proof}
We will use the interpretation of the elements of $C^\bullet (A \oplus M, A \oplus M)$ as coderivations on the free tensor coalgebra $\overline{T}^c (A \oplus M)$ and the Gerstenhaber bracket on $C^\bullet (A \oplus M, A \oplus M)$ as the commutator bracket of coderivations on $\overline{T}^c (A \oplus M)$. Let $\overline{\mu}$, $\overline{\eta}$ and $\overline{T}$ denote the coderivations on $\overline{T}^c ( A \oplus M)$ corresponding to the maps $\sum_{k \geq 1} \mu_k$, $\sum_{k \geq 1} \eta_k$ and $\sum_{k \geq 1} T_k$ in $C^\bullet (A \oplus M, A \oplus M).$ Hence $T$ is a Maurer-Cartan element if and only if 
\begin{align}\label{rel-hom-1}
\sum_{n=1}^\infty \frac{1}{n!} ~\overline{P} [ \cdots [[ \overline{\mu}+ \overline{\eta}, \overline{T}], \overline{T}], \ldots, \overline{T}] = 0,
\end{align}
where $\overline{P}$ denotes the projection onto the abelian subalgebra $\bigoplus_{n \in \mathbb{Z}} \mathrm{Hom}_n (\overline{T}(M), A)$. First observe that 
\begin{align}\label{rel-hom-2}
&\overline{P} [ \cdots [[ \overline{\mu}+ \overline{\eta}, \overline{T}], \overline{T}], \ldots, \overline{T}] (u_1, \ldots, u_p)  \nonumber \\
&= \mathrm{pr}_A \sum_{i=0}^n (-1)^i {n \choose i} (\underbrace{\overline{T} \circ \cdots \circ \overline{T}}_{i} \circ (\overline{\mu} + \overline{\eta}) \circ \underbrace{\overline{T} \circ \cdots \circ \overline{T}}_{n-i})(u_1, \ldots, u_p ) \nonumber \\
&= \big( (\mathrm{pr}_A \circ \overline{\mu} \circ \underbrace{\overline{T} \circ \cdots \circ \overline{T}}_{n} )  ~-~ n (\mathrm{pr}_A \circ \overline{T} \circ \overline{\eta} \circ \underbrace{\overline{T} \circ \cdots \circ \overline{T}}_{n-1}) \big)(u_1, \ldots, u_p ) \qquad
(\text{the other terms are zero}).
\end{align}
By a straightforward calculation, we obtain
\begin{align*}
&(\mathrm{pr}_A \circ \overline{\mu} \circ \underbrace{\overline{T} \circ \cdots \circ \overline{T}}_{n} ) (u_1, \ldots, u_p ) \\
&= \sum_{k_1 + \cdots + k_n = p} ~\mu_n \big( T_{k_1} (u_1, \ldots, u_{k_1}), T_{k_2} (u_{k_1 + 1}, \ldots, u_{k_1 + k_2}), \ldots, T_{k_n } (u_{k_1 + \cdots + k_{n-1} +1} , \ldots, u_{p})   \big),
\end{align*}
and
\begin{align*}
&n (\mathrm{pr}_A \circ \overline{T} \circ \overline{\eta} \circ \underbrace{\overline{T} \circ \cdots \circ \overline{T}}_{n-1})  (u_1, \ldots, u_p ) \\
&= n \sum_{i=1}^n \sum_{ \substack{j+k_1+ \cdots + k_{i-1} +1 \\ + k_{i+1} + \cdots + k_n + t = p}} (-1)^{|u_1| + \cdots + |u_j|}  ~ T_{j+1+t} \bigg(   u_1, \ldots, u_j , \eta \big(   T_{k_1} ( u_{j+1} , \ldots u_{j+ k_1} ), \ldots, 
T_{k_{i-1}} ( \cdots , u_{j+ k_1 + \cdots + k_{i-1}} ),\\ 
&  \qquad \qquad \qquad \qquad u_{j + k_1 + \cdots + k_{i-1} +1}, T_{k_{i+1}} (\cdots, u_{j+k_1+ \cdots + k_{i-1} +1 + k_{i+1}}), \ldots, T_{k_n} (\ldots, u_{p-t}) \big), u_{p-t+1}, \ldots, u_p \bigg).
\end{align*}
Therefore, it follows from equations (\ref{rel-hom-1}) and (\ref{rel-hom-2}) that $T$ is a homotopy relative Rota-Baxter operator if and only if the identity (\ref{homo-rota-iden}) holds.
\end{proof}

\begin{remark}
The identities given by equation (\ref{homo-rota-iden}) can be alternatively used as a definition of a homotopy relative Rota-Baxter operator. A homotopy Rota-Baxter operator on an $A_\infty [1]$-algebra $(A, \{ \mu_k \}_{k \geq 1})$ is a homotopy relative Rota-Baxter operator on the adjoint bimodule.
\end{remark}

\begin{defn}
(i) A homotopy relative Rota-Baxter algebra is a triple $((A, \{\mu_k \}_{k \geq 1}), ( M, \{ \gamma_k \}_{k \geq 1}), \{ T_k \}_{k \geq 1} )$ consisting of an $A_\infty[1]$-algebra, a bimodule and a homotopy relative Rota-Baxter operator.

(ii) A homotopy Rota-Baxter algebra is a pair $((A, \{\mu_k \}_{k \geq 1}), \{ T_k \}_{k \geq 1} )$ of an $A_\infty [1]$-algebra and a homotopy Rota-Baxter operator on it.
\end{defn}

\medskip

In \cite{laza-rota}, the authors introduce homotopy relative Rota-Baxter operators on modules over $L_\infty [1]$-algebra. Here we first recall their definition and then we compare with homotopy relative Rota-Baxter operator on bimodules over $A_\infty [1]$-algebras. Let $(W, \{ l_k \}_{k \geq 1} )$ be an $L_\infty [1]$-algebra and $( V, \{ \rho_k \}_{k \geq 1})$ be a representation of it. Then 
\begin{align*}
\big( L = \bigoplus_{ n \in \mathbb{Z}} \mathrm{Hom}_n( \overline{S}( W \oplus V), W \oplus V) , \mathfrak{a} = \bigoplus_{n \in \mathbb{Z}} \mathrm{Hom}_n (\overline{S}(V), W), P, \triangle = \sum_{k \geq 1} (l_k + \rho_k)   \big)
\end{align*}
is a $V$-data, where the graded vector space $L$ is equipped with the standard Nijenhuis-Richardson bracket \cite{nij-ric,laza-rota}. Hence by Theorem \ref{Voro-JPAA1}, the graded space $\mathfrak{a} = \bigoplus_{n \in \mathbb{Z}} \mathrm{Hom}_n (\overline{S}(V), W)$ carries an $L_\infty [1]$-algebra structure which turns out to be weakly filtered. A Maurer-Cartan element $T = \sum_{ k \geq 1} T_k \in \mathrm{Hom}( S(V), W)$ on this $L_\infty [1]$-algebra is said to be a homotopy relative Rota-Baxter operator on the module $( V, \{ \rho_k \}_{k \geq 1})$ over the $L_\infty [1]$-algebra $(W, \{ l_k \}_{k \geq 1} )$. A homotopy relative Rota-Baxter operator $T = \sum_{k \geq 1} T_k$ is said to be strict if $T_k = 0$ for $k \geq 2$. It has been shown in \cite{laza-rota} that a strict homotopy relative Rota-Baxter operator $T = T_1$ induces a $pre\text{-}Lie_\infty [1]$-algebra structure on $V$ given by
\begin{align*}
\theta_k (v_1, \ldots, v_k ) := \rho_k ( Tv_1, \ldots, Tv_{k-1}, v_k ), ~ \text{ for } k \geq 1.
\end{align*}

\medskip

Let $(A, \{ \mu_k \}_{k \geq 1})$ be an $A_\infty [1]$-algebra and $(M, \{ \eta_k \}_{k \geq 1})$ be a bimodule over it. It has been proved in \cite{lada-markl} that the standard symmetrization process yields an $L_\infty [1]$-algebra $(A, \{l_k \}_{k \geq 1})$ and a representation $(M, \{ \rho_k\}_{k \geq 1})$ of it, where
\begin{align}
l_k ( a_1, \ldots, a_k ) :=~& \sum_{\sigma \in S_k }  \epsilon (\sigma ) \mu_k ( a_{\sigma (1)}, \ldots, a_{\sigma (k)}), ~ \text{ for } k \geq 1 \text{ and } a_1, \ldots, a_k \in A, \label{symm-l}\\
\rho_k ( a_1, \ldots, a_k ) :=~& \sum_{\sigma \in S_k }  \epsilon (\sigma ) \eta_k ( a_{\sigma (1)}, \ldots, a_{\sigma (k)}), ~ \text{ for } k \geq 1 \text{ and } a_1, \ldots, a_{k-1} \in A, a_k \in M. \label{symm-mod}
\end{align}
Moreover, the above symmetrization process yields a morphism of graded Lie algebras
\begin{align*}
\Psi : \bigoplus_{n \in \mathbb{Z}} \mathrm{Hom}_n ( \overline{T}(A \oplus M), A \oplus M)  \longrightarrow  \bigoplus_{n \in \mathbb{Z}} \mathrm{Hom}_n ( \overline{S}(A \oplus M), A \oplus M).
\end{align*}
As a consequence, we get the following.
\begin{prop}
With the above notations, if $T$ is a homotopy relative Rota-Baxter operator on $(M, \{\eta_k\}_{k\geq 1})$ over the $A_\infty[1]$-algebra $(A, \{ \mu_k \}_{k \geq 1})$, then $\Psi (T)$ is a homotopy relative Rota-Baxter operator on the module $(M, \{ \rho_k\}_{k \geq 1})$ over the $L_\infty [1]$-algebra $(A, \{ l_k \}_{k \geq 1})$.
\end{prop}

\subsection{Homotopy dendriform and homotopy pre-Lie algebras}
Dendriform algebras were introduced by Loday \cite{loday} as a Koszul dual of associative dialgebras. The notion of homotopy dendriform algebras ($Dend_\infty$-algebras) was defined in \cite{lod-val-book} and explicitly described in \cite{das1}. Here, we recall the definition of a $Dend_\infty[1]$-algebra from \cite{das1} in a slightly different way using Maurer-Cartan elements in a graded Lie algebra.

Let $C_k$ be the set of first $n$ natural numbers. For convenience, we denote the elements of $C_k$ by $C_k = \{ [1], [2], \ldots, [k] \}$. There are some distinguished maps $R_0 (k; \overbrace{1, \ldots, \underbrace{l}_{i\text{-th}}, \ldots, 1}^k) : C_{k+l-1} \rightarrow C_k$ and $R_i (k; \overbrace{1, \ldots, \underbrace{l}_{i\text{-th}}, \ldots, 1}^k) : C_{k+l-1} \rightarrow \mathbb{K}[C_l]$ given by

\begin{center}
\begin{tabular}{|c|c|c|c|}
\hline
$[r] \in C_{k+l-1}$ & $1 \leq r \leq i-1$ & $i \leq r \leq i+l-1$ & $i+l \leq r \leq k+l-1$\\ \hline \hline
$R_0 (k; 1, \ldots, l, \ldots, 1)[r]$ & $[r]$ & $[i]$ & $[r-l+1]$\\ \hline
$ R_i (k; 1, \ldots, l, \ldots, 1)[r]$ & $[1] + \cdots + [l]$  & $[r-i+1]$ & $[1] + \cdots + [l]$ \\ \hline
\end{tabular}
\end{center}

\medskip

\medskip

Let $D = \bigoplus_{i \in \mathbb{Z}} D_i$ be a graded vector space. Suppose $\mu_k : TD \otimes D \otimes TD \rightarrow D$ is a map which is nonzero only on $\bigoplus_{i+j = k+1, i, j \geq 1} D^{\otimes i-1} \otimes D \otimes D^{\otimes j-1}$. Then $\mu_k$ induces $k$ many maps $\mu_{k, [1]}, \ldots, \mu_{k, [k]} : D^{\otimes k } \rightarrow D$ by
\begin{align*}
\mu_{k, [1]}( a_1, \ldots, a_k) =&  \mu_k ( 1 \otimes a_1 \otimes a_1 \cdots a_k), \quad
\mu_{k, [2]} (a_1, \ldots, a_k) = \mu_k ( a_1 \otimes a_2 \otimes a_3 \cdots a_k), ~ \ldots \\
&\mu_{k, [k]} (a_1, \ldots, a_k) = \mu_k (a_1 \cdots a_{k-1} \otimes a_k \otimes 1).
\end{align*}
Consider the graded space $\bigoplus_{n \in \mathbb{Z}} \mathrm{Hom}_n ( TD \otimes D \otimes TD , D)$ of such maps. It carries a graded Lie bracket defined by $[\mu_k , \eta_l] = \mu_k \circ \eta_l - (-1)^{mn} \eta_l \circ \mu_k$, where
\begin{align*}
&(\mu_k \circ \eta_l ) ( a_1 \cdots a_{r-1} \otimes a_r \otimes a_{r+1} \cdots a_{k+l-1}) \\
&= \sum_{i =1}^k (-1)^{|a_1|+ \cdots + |a_{i -1}|}~ \mu_{k , R_0 (k; 1, \ldots, l,\ldots, 1 )[r]} ( a_1, \ldots, a_{i-1}, \mu_{l , R_i (k; 1, \ldots, l,\ldots, 1 )[r]} (a_i, \ldots, a_{i+l-1}), \ldots, a_{k+l-1}),
\end{align*}
for $\mu = \sum_{k \geq 1} \mu_k \in \mathrm{Hom}_m (TD \otimes D \otimes TD, D)$ and $\eta = \sum_{l \geq 1} \eta_l \in \mathrm{Hom}_n (TD \otimes D \otimes TD, D).$
A $Dend_\infty [1]$-algebra is a graded vector space $D = \bigoplus_{i \in \mathbb{Z}} D_i$ together with a Maurer-Cartan element $\mu = \sum_{k \geq 1} \mu_k$ of the graded Lie algebra $(\bigoplus_{n \in \mathbb{Z}} \mathrm{Hom}_n ( TD \otimes D \otimes TD, D), [~,~])$. Note that each $\mu_k$ determines $k$ many degree $1$ maps $D^{\otimes k} \rightarrow D$. Therefore, for now onward, we will denote a $Dend_\infty$-algebra by $(D, \{ \mu_{k, [r] } \}_{k \geq 1, [r] \in C_k})$. If $(D, \{ \mu_{k, [r] } \}_{k \geq 1, [r] \in C_k} )$ is a $Dend_\infty [1]$-algebra, then it has been proved in \cite{das1,lod-val-book} that $(D, \{\mu_k \}_{k \geq 1})$ is an $A_\infty [1]$-algebra, where
\begin{align}\label{dend-sum-a}
\mu_k = \mu_{k, [1]} + \cdots + \mu_{k, [k]}, ~ \text{ for } k \geq 1.
\end{align}
%\end{defn}

The notion of $pre\text{-}Lie_\infty [1]$-algebras are defined by Chapoton and Livernet \cite{chap-mur}. Here we will not recall the definition rather mention that there is a graded Lie bracket $[~,~]_{\text{MN}}$, known as {\em Matsushima-Nijenhuis bracket} on $\bigoplus_{n \in \mathbb{Z}} \mathrm{Hom}_n ( S(W) \otimes W, W)$, for any graded vector space $W$. See \cite{chap-mur,laza-rota} for details. A $pre\text{-}Lie_\infty [1]$-algebra structure on a graded vector space $W$ is by definition a Maurer-Cartan element in the graded Lie algebra $\big( \bigoplus_{n \in \mathbb{Z}} \mathrm{Hom}_n ( S(W) \otimes W, W), [~,~]_{\text{MN}} \big).$

A $pre\text{-}Lie_\infty[1]$-algebra induces an $L_\infty [1]$-algebra structure on $W$ with structure maps given by
\begin{align*}
l_k (x_1, \ldots, x_k ) = \sum_{i=1}^k (-1)^{|x_i| (|x_{i+1}| + \cdots + |x_k|)} ~ \theta_k (x_1, \ldots, \widehat{x_i}, \ldots, x_k, x_i ), ~ \text{ for } k \geq 1.
\end{align*}
This is called the subadjacent $L_\infty [1]$-algebra of the $pre\text{-}Lie_\infty[1]$-algebra.

It is known that a dendriform algebra gives rise to a pre-Lie algebra structure. Here we prove a homotopy version of this result.

\begin{thm}\label{last-thm}
Let $(D, \{ \mu_{k, [r]} \}_{k \geq 1, [r] \in C_k} )$ be a $Dend_\infty [1]$-algebra. Then $(D, \{ \theta_k \}_{k \geq 1})$  is a $pre\text{-}Lie_\infty[1]$-algebra, where
\begin{align}\label{dend-pre-map}
\theta_k ( a_1, \ldots, a_k ) = \sum_{\sigma \in S_k} \epsilon (\sigma) \mu_{k, [ \sigma^{-1} (k)]} ( a_{\sigma (1)}, \ldots, a_{\sigma (k)}), ~ \text{ for } k \geq 1.
\end{align}
\end{thm}

\begin{proof}
We define a map
\begin{align*}
\Psi : \bigoplus_{n \in \mathbb{Z}} \mathrm{Hom}_{n} ( TD \otimes D \otimes TD, D) &\longrightarrow \bigoplus_{n \in \mathbb{Z}} \mathrm{Hom}_{n} (\overline{S}(D) \otimes D , D) ~~~ \text{ by }\\
\Psi ( \mu_k ) (a_1, \ldots, a_k) &= \sum_{\sigma \in S_k} \epsilon (\sigma) \mu_{k, [ \sigma^{-1} (k)]} ( a_{\sigma (1)}, \ldots, a_{\sigma (k)}).
\end{align*}
It is a straightforward but tedious calculation to check that $\Psi$ preserves the graded Lie algebra structures. Hence the result follows from the definition of $pre\text{-}Lie_\infty[1]$-algebra structure by Maurer-Cartan element.
\end{proof}

The above construction is functorial with respect to strict morphisms in the categories of $Dend_\infty [1]$-algebras and $pre\text{-}Lie_\infty[1]$-algebras.

\begin{prop} With the above construction, the following diagram commutes
\begin{align}
\xymatrix{
 Dend_\infty[1] \ar@{~>}[rr] \ar@{~>}[d]& & pre\text{-}Lie_\infty [1]  \ar@{~>}[d]\\
 A_\infty [1] \ar@{~>}[rr]& & L_\infty [1].
}
\end{align}
\end{prop}

\begin{proof}
Let $(D, \{ \mu_{k, [r]} \}_{k \geq 1, [r] \in C_k})$ be a $Dend_\infty$-algebra with the corresponding $A_\infty [1]$-algebra $(D, \{ \mu_k \}_{k \geq 1})$ where $\mu_k$'s are given by equation (\ref{dend-sum-a}). Therefore, the corresponding $L_\infty [1]$-algebra structure on $D$ is given by 
\begin{align}\label{l-inf-1}
l_k (a_1, \ldots, a_k ) =~& \sum_{\sigma \in S_k} \epsilon (\sigma) ~\mu_{k} ( a_{\sigma (1)}, \ldots, a_{\sigma (k)}) \nonumber \\
=~& \sum_{\sigma \in S_k} \epsilon (\sigma) \sum_{i=1}^k \mu_{k, [i]} ( a_{\sigma (1)}, \ldots, a_{\sigma (k)}).
\end{align}
On the other hand, consider the $pre\text{-}Lie_\infty [1]$-algebra $(D, \{ \theta_k \}_{k \geq 1})$ where $\theta_k$'s are given in equation \eqref{dend-pre-map}. Hence, the subadjacent $L_\infty [1]$-algebra on $D$ has structure maps
\begin{align}\label{l-inf-2}
l_k' (a_1, \ldots, a_k ) =~& \sum_{i=1}^k (-1)^{|a_i| (|a_{i+1}| + \cdots + |a_k|)}~ \theta_k (a_1, \ldots, \widehat{ a_i}, \ldots, a_k, a_i ) \nonumber \\
=~& \sum_{i=1}^k (-1)^{|a_i| (|a_{i+1}| + \cdots + |a_k|)}~ \sum_{\sigma \in S_k} \epsilon ( \sigma \tau)~ \mu_{k, [\tau^{-1} \sigma^{-1} (k)]} (a_{\sigma (1)}, \ldots, a_{\sigma (k)}) \nonumber \\
=~& \sum_{i=1}^k \sum_{\sigma \in S_k} \epsilon ( \sigma)~ \mu_{k, [\tau^{-1} \sigma^{-1} (k)]} (a_{\sigma (1)}, \ldots, a_{\sigma (k)}).
\end{align}
In the second equality, given a fixed $i$, we use the permutation $\tau \in S_k$ given by 
\begin{align*}
\tau (j) = j \text{ for } j \leq i-1, \quad \tau (j) = j+1 \text{ for }   i \leq j \leq k-1 \quad \text{ and } \quad \tau (k) = i.
\end{align*}
Finally, the third equality follows as $\epsilon (\tau ) = (-1)^{|a_i| (|a_{i+1}| + \cdots + |a_k|)}$. From equations (\ref{l-inf-1}) and (\ref{l-inf-2}), it follows that $l_k = l_k'$ for $k \geq 1$. Hence the result follows.
\end{proof}

\subsection{Strict homotopy relative Rota-Baxter operators}

A homotopy relative Rota-Baxter operator $T = \sum_{k \geq 1} T_k$ on a bimodule $(M , \{ \eta_k \}_{k \geq 1})$ over an $A_\infty [1]$-algebra is said to be {\em strict} if $T_k = 0$ for $k \geq 2$.
It follows from equation (\ref{homo-rota-iden}) that a degree $0$ linear map $T: M \rightarrow A$ is a strict homotopy relative Rota-Baxter operator if $T$ satisfies
\begin{align*}
\mu_k ( Tu_1, \ldots, Tu_k) = \sum_{r=1}^k ~ T( \eta_k ( Tu_1, \ldots, u_r, \ldots, Tu_k )), ~ \text{ for } k \geq 1.
\end{align*}
Strict homotopy relative Rota-Baxter operators are considered in \cite{das1} and the following result is proved in the same reference.
\begin{prop}
Let $(A, \{ \mu_k \}_{k \geq 1})$ be an $A_\infty [1]$-algebra and $(M, \{ \eta_k \}_{k \geq 1})$ be a bimodule. If $T$ is a strict homotopy relative Rota-Baxter operator, then $M$ carries a $Dend_\infty [1]$-algebra structure, where
\begin{align}\label{rota-dend-inf-brk}
\mu_{k, [r]} ( u_1, \ldots, u_k ) = \eta_k (  Tu_1, \ldots, u_r, \ldots, Tu_k ), ~ \text{ for } k \geq 1 \text{ and } [r] \in C_k.
\end{align}
\end{prop}
The $Dend_\infty [1]$-algebra constructed in the above proposition is called the subadjacent $Dend_\infty [1]$-algebra associated to the strict homotopy relative Rota-Baxter operator $T$.

\begin{prop}
The following diagram commutes
\begin{align}
\xymatrix{
 \mathrm{strict~ homotopy~ relative~ Rota-Baxter ~algebra} \ar@{~>}[rr] \ar@{~>}[d]& & Dend_\infty [1]  \ar@{~>}[d]\\
 \mathrm{strict~ homotopy~ relative~ Rota-Baxter~ Lie~ algebra} \ar@{~>}[rr]& & pre\text{-}Lie_\infty [1].
}
\end{align}
\end{prop}

\begin{proof}
Let $T: M \rightarrow A$ be a strict homotopy relative Rota-Baxter operator on the bimodule $(M, \{ \eta_k \}_{k \geq 1})$ over the $A_\infty [1]$-algebra $(A, \{ \mu_k \}_{k \geq 1})$. Let $(M, \{ \mu_{k, [r]} \}_{k \geq 1, [r] \in C_k} )$ denote the corresponding $Dend_\infty [1]$-algebra, where $\mu_{k, [r]}$'s are given by equation (\ref{rota-dend-inf-brk}). Hence by Theorem \ref{last-thm}, the corresponding $pre\text{-}Lie_\infty [1]$-algebra structure on $M$ are given by
\begin{align*}
\theta_k (u_1, \ldots, u_k) =~& \sum_{\sigma \in S_k} \epsilon ( \sigma) \mu_{k, [\sigma^{-1} (k)]} ( u_{\sigma (1)}, \ldots, u_{\sigma (k)}) \\
=~& \sum_{\sigma \in S_k} \epsilon ( \sigma) \eta_{k} ( Tu_{\sigma (1)}, \ldots, u_{\sigma (j)}, \ldots, u_{\sigma (k)}) |_{\sigma (j) =k}.
\end{align*}
On the other hand, by considering $T$ as a strict homotopy relative Rota-Baxter operator on the module $(M, \{ \rho_k \}_{k \geq 1})$ over the $L_\infty [1]$-algebra $(A, \{ l_k \}_{k \geq 1})$ given in equations (\ref{symm-l}) and (\ref{symm-mod}), the $pre\text{-}Lie_\infty [1]$-algebra on $M$ is given by
\begin{align*}
\theta_k' (u_1, \ldots, u_k ) = \rho_k ( Tu_1, \ldots, Tu_{k-1}, u_k )
= \sum_{\sigma \in S_k} \epsilon (\sigma) \eta_k ( Tu_{\sigma (1)}, \ldots, u_{\sigma (j)}, \ldots, Tu_{\sigma (k)} )|_{\sigma (j) = k}.
\end{align*} 
Thus we have $\theta_k = \theta_k'$, for $k \geq 1$ which proves the result.
\end{proof}

\noindent {\em Acknowledgements.} The research of A. Das is supported by the postdoctoral fellowship of Indian Institute of Technology (IIT) Kanpur, and the research of S. K. Mishra is supported by the NBHM postdoctoral fellowship. Both the authors thank their funding institute/organisation for their support.

\end{document}